\documentclass[9pt]{amsart}
\usepackage{amsmath}
\usepackage{amsfonts}
\usepackage{color}
\usepackage{verbatim}
\usepackage{eepic}
\usepackage{epic}
\usepackage{epsfig,subfigure}
\usepackage[]{graphics}
\usepackage{geometry}

\usepackage[colorlinks,linktocpage,linkcolor=blue]{hyperref}

\def\tribar{\vert\thickspace\!\!\vert\thickspace\!\!\vert}

\allowdisplaybreaks[1]

\newtheorem{theorem}{Theorem}[section]
\newtheorem{lemma}{Lemma}[section]

\newtheorem{remark}{Remark}[section]

\numberwithin{equation}{section}

\begin{document}

\newcommand{\al}{\alpha}
\newcommand{\be}{\beta}
\newcommand{\om}{\omega}
\newcommand{\ga}{\gamma}
\newcommand{\wh}{\widehat}
\newcommand{\Ga}{\Gamma}
\newcommand{\ka}{\kappa}
\newcommand{\teps}{\tilde e}
\newcommand{\fy}{\varphi}
\newcommand{\Om}{\Omega}
\newcommand{\si}{\sigma}
\newcommand{\Si}{\Sigma}
\newcommand{\de}{\delta}
\newcommand{\De}{\Delta}
\newcommand{\la}{\lambda}
\newcommand{\La}{\Lambda}
\newcommand{\ep}{\epsilon}

\newcommand{\vth}{\vartheta}
\newcommand{\vtht}{{\widetilde \vartheta}}
\newcommand{\rh}{\varrho}
\newcommand{\rlh}{{\widetilde \varrho}}

\renewcommand{\for}{\quad{\mbox{for}}\quad}
\newcommand{\on}{\quad\text{on}\ }
\newcommand{\orr}{\quad\text{or}\ }
\newcommand{\inn}{\quad\text{in}\ }
\newcommand{\as}{\quad\text{as}\ }
\newcommand{\at}{\quad\text{at}\ }
\newcommand{\ifff}{\quad\text{if}\ }
\newcommand{\andy}{\quad\text{and}\ }
\newcommand{\with}{\quad\text{with}\ }
\newcommand{\when}{\quad\text{when}\ }
\newcommand{\where}{\quad\text{where}\ }
\newcommand{\FEM}{\text{finite element method }}
\newcommand{\FE}{\text{finite element }}
\def\tribar{\vert\thickspace\!\!\vert\thickspace\!\!\vert}

\def\Mitaga1{{E_{\al,1}}}
\def\Mitagaa{{E_{\al,\al}}}
\def\C{{\mathbb{C}}}
\def\V{{H^1_0(\Omega)}}

\def\NN{{N}}
\def\Etilh{{\bar{E}_h}}

%
\def\guh{{\underline u}_h}    
\def\guht{{\underline u}_{h,t}}    
\def\puht{u_{h,t}}                 
\def\puh{\widetilde{u}_h}                 
\def\luh{{\bar u}_h}     
\def\luht{{\bar u}_{h,t}}
\def\luhtt{{\bar u}_{h,tt}}
\def\p{q}
\def\q{\luht}
\def\luhp{\luht}
\def\L2o{{L_2(\Om)}}
\def\K{\tau}
\def\zK{z^\K}
\def\T{{\mathcal{T}}}

\def\Dal{{\partial^\alpha_t}}
\def\Dt{{\partial_t}}
\def\Dalpl{{\partial^{\alpha+l}_t}}
\def\gamal{{\gamma}_\alpha}

\def\bDelh{{\bar{\Delta}_h}}
\def\Ftilh{\bar{F}_h}
\def\Pbarh{\bar{P}_h}
\def\Ebar{{F}}
\def\Ellipt{{\mathcal L}}
\def\dtoa{D_t^{1-\al}}
\def\dta{D_t^{\al}}
\def\adtoa{(D_t^{1-\al})^*}
\def\oatwo{{(1-\alpha)/2}}
\def\adta{(D_t^{\al})^*}
\def\dH#1{\dot H^{#1}(\Omega)}
\def\normh#1#2{\tribar #1 \tribar_{\dot H^{#2}(\Omega)}}

\title[Galerkin FEM for inhomogeneous time-fractional diffusion]
{Error Analysis of Semidiscrete Finite Element Methods for Inhomogeneous Time-Fractional Diffusion}
\author {Bangti Jin \and Raytcho Lazarov \and Joseph Pasciak \and Zhi Zhou}
\address {Department of Mathematics and Institute for
Applied Mathematics and Computational Science, Texas A\&M University,
College Station, TX 77843-3368 ({\texttt{btjin,lazarov,pasciak,zzhou@math.tamu.edu}})}

\date{started July, 2012; today is \today}

\maketitle

\begin{abstract}
We consider the initial boundary value problem for the inhomogeneous time-fractional diffusion equation with a homogeneous
Dirichlet boundary condition and a nonsmooth right hand side data in a bounded convex polyhedral domain.
We analyze two semidiscrete schemes based on the standard Galerkin and lumped mass finite element methods.
Almost optimal error estimates are obtained for right hand side data $f(x,t)\in L^\infty(0,T;\dot H^q(\Omega))$,
$-1< q \le 1$, for both semidiscrete schemes. For lumped mass method, the optimal $L^2(\Omega)$-norm error estimate
requires symmetric meshes. Finally, numerical experiments for one- and two-dimensional examples are presented to verify our theoretical results.
\end{abstract}

\section{Introduction}\label{sec:intro}
We consider the model initial-boundary value problem for the fractional-order parabolic
differential equation (FPDE) for $u(x,t)$:
\begin{alignat}{3}\label{eq1}
   \Dal u-\Delta u&= f,&&\quad \text{in  } \Omega&&\quad T \ge t > 0,\notag\\
   u&=0,&&\quad\text{on}\  \partial\Omega&&\quad T \ge t > 0,\\
    u(0)&=v,&&\quad\text{in  }\Omega,&&\notag
\end{alignat}
where $\Omega$ is a bounded convex polygonal domain in $\mathbb R^d\,(d=1,2,3)$ with
a boundary $\partial\Omega$, $v$ and $f$ are given functions and $T>0$ is a fixed value.
In the model \eqref{eq1}, $\Dal u$ refers to the Caputo fractional derivative of order
$\alpha$ ($0<\alpha<1$) of the function $u(t)$,
and it is defined by \cite[pp. 91, eq. (2.4.1)]{KilbasSrivastavaTrujillo:2006}
\begin{equation*}
  \Dal u (t) = \frac{1}{\Gamma(1-\alpha)}\int_0^t \frac{1}{(t-s)^\alpha}u'(s)ds.
\end{equation*}
It is known that for the fractional order $\alpha=1$, the fractional derivative $\Dal u$
recovers the canonical first-order derivative $u'(t)$ \cite[pp. 92, eq. (2.4.14)]{KilbasSrivastavaTrujillo:2006},
and accordingly the model \eqref{eq1} reduces to the classical time-dependent diffusion
problem. Therefore, the model \eqref{eq1} can be regarded as a fractional
counterpart of the standard diffusion problem.

The interest in \eqref{eq1} is mainly motivated by anomalous diffusion processes,
 known as subdiffusion, in which the mean square variance grows slower
than that in a Gaussian process. At a microscopic level, the diffusion process
results from random motion of particles. In a subdiffusion process, the waiting time
between consecutive random particle motion can be very large, and thus the mean
waiting time diverges. Thus, the underlying stochastic process deviates significantly
from the Brownian motion, and instead it can only be more adequately described by
the continuous time random walk (CTRW) \cite{MontrollWeiss:1965}. This microscopic
explanation has been frequently exploited in applications. For example, Adams and Gelhar \cite{AdamsGelhar:1992} observed
that field data show anomalous diffusion in a highly heterogeneous aquifer, and
Hatano and Hatano \cite{HatanoHatano:1998} applied the CTRW to model
anomalous diffusion in an underground environmental problem. The macroscopic
counterpart of the CTRW is a diffusion equation with a time fractional derivative
$\Dal u(t)$, i.e. model \eqref{eq1}. It has been successfully applied to many practical
problems, including diffusion in media with fractal geometry \cite{Nigmatulin},
the dynamics of viscoelastic materials \cite{GionaCerbeliRoman:1992},
and contaminant transport within underground water flow \cite{Berkowitz:2002}.

The modeling capabilities of FPDEs have generated considerable
interest in deriving, analyzing, and testing numerical methods for such problems. As
a result, a number of numerical techniques were developed, and their stability and
convergence properties were investigated. Yuste and Acedo \cite{YusteAcedo:2005} presented
a numerical scheme by combining the forward time centered space
method for the ordinary diffusion equation and
Grunwald-Letnikov discretization of the (Riemann-Liouville type) fractional-derivative and
provided a von Neumann type stability analysis. Lin and Xu \cite{LinXu:2007}
proposed a numerical method based a finite difference scheme in time
and Legendre spectral method in space, showed its unconditional stability,
and provided error estimates. Li and Xu \cite{LiXu:2009} developed a spectral
method in both temporal and spatial variables and established various a priori error
estimates. Mustapha \cite{Mustapha:2011} studied semidiscrete in time and fully discrete schemes and
derived error bounds for smooth initial data \cite[Theorem 4.3]{Mustapha:2011}.
In all these useful studies, the error analysis was carried out under the assumption that the solution
is sufficiently smooth. The optimality of the estimates with respect to
the solution smoothness expressed through the problem data, i.e., the right hand side
$f$ and the initial data $v$, was not considered.

Thus, these useful studies do not cover the interesting case of solutions with limited regularity due to
low regularity of the data, a typical case for related inverse problems; see, e.g.,
\cite{ChengNakagawaYamamoto:2009} and also \cite{KeungZou:1998} and \cite{XieZou:2006} for its
parabolic counterpart. In our earlier work \cite{Bangti_LZ_2013}, we have analyzed the semidiscrete
Galerkin finite element method (FEM) and lumped mass method for problem \eqref{eq1} with a zero
right hand side. In particular, almost optimal error estimates were established for both smooth
and nonsmooth initial data, i.e., $v\in \dH q$, $0\leq q\leq 1$.  (See section 2.2 for the
definitions of the space $\dH q$.) More recently in \cite{JinLazarovPasciakZhou:2013},
the results were generalized to the case of very weak initial data, i.e., $v\in\dH q$, $-1<q<0$ .
We also refer interested readers to \cite{McLean-Thomee-IMA,McLeanThomee:2010}
for related studies on FPDEs with a Riemann-Liouville fractional derivative and non-smooth initial data.

In this work, we analyze the case of a non-smooth right hand side, i.e., $f\in L^r(0,T;\dH q)$,
$-1<q\leq1$,  $r>1$. Such problems occur in many practical applications, e.g.,
optimal control problems and inverse problems \cite{JinRundell:2012b}.
Thus, it is of immense interest to develop and to analyze related
numerical schemes. However, with such weak data it might be difficult to define
a proper weak solution. In Remark \ref{weakest_sol} below, we note that
a function $u(x,t)$ expressed by the representation \eqref{Duhamel2}
satisfies the differential equation for $1< r$, but it satisfies (in generalized sense)
the initial condition $u(x,0)=0$ for $r >1/\al$ only. Therefore, the natural
class of weak data would be $f\in L^r(0,T;\dH q)$, $-1<q\leq1$, $r>1/\al$.
In fact, all results (upon minor modifications) in this paper are valid for
problem \eqref{eq1}  with such data. However, for the ease of exposition,
we shall assume that $f \in L^\infty(0,T;\dH q)$, $-1<q\leq1$,
which guarantees that the representation formula \eqref{Duhamel2} does give a
legitimate solution weak solution $u(x,t)$ for all $0 < \al <1$.

The goal of this paper is to develop an error analysis with optimal with respect
to the regularity of the right hand side estimates for the semidiscrete Galerkin and the
lumped mass Galerkin FEM for problem \eqref{eq1} on convex
polygonal domains. Now we describe the numerical schemes, using the standard notation from \cite{Thomee97}.
Let $\{\mathcal{T}_h\}$, $0<h<1$, be a family of shape regular and quasi-uniform partitions
of the domain $\Omega$ into $d$-simplices, called finite elements, with $h$ denoting the
maximum diameter. The approximate solution $u_h$
will be sought in the finite element space $X_h=X_h(\Omega)$
of continuous piecewise linear functions over the triangulation $\mathcal{T}_h$
\begin{equation*}
   X_h = \{\chi\in H_0^1(\Omega): \chi \mbox{ is a linear function over }  \tau \quad \forall \tau \in \mathcal{T}_h\}.
\end{equation*}

The semidiscrete Galerkin FEM for problem \eqref{eq1} reads: find $u_h(t)\in X_h$
such that
\begin{equation*}
  (\Dal u_h(t), \chi) + a(u_h, \chi) = (f,\chi),\quad \forall \chi\in X_h, \ \ t>0,
\end{equation*}
with $u_h(0)=0$, and $a(u,w) = (\nabla u, \nabla w)$ for $u,w\in H_0^1(\Omega)$.
We shall study the convergence of the semidiscrete Galerkin FEM for the case
of a right hand side $f\in L^\infty(0,T;\dH q)$, $-1<q\leq1$.

The main difficulty in the analysis stems from
limited smoothing properties of the solution operator,  cf. Lemma \ref{lem:E}.
This difficulty  is overcome by
exploiting the mapping property of discrete solution operators established in Lemma \ref{lem:Eh}.
Our main results are as follows. First in Theorem \ref{thm:gal:l2}, we derive an optimal
$L^2(0,T;\dot H^p(\Omega))$-norm, $p=0,1$, error bound:
\begin{equation*}
 \| u_h - u \|_{L^2(0,T;L^2(\Om))} +
h\|\nabla(u_h- u)\|_{L^2(0,T;L^2(\Om))} \le Ch^{2+\min(q,0)}  \|f\|_{L^2(0,T;\dH q)},
\end{equation*}
for $-1<q\leq1$. Second, we derive an almost optimal $L^\infty(0,T;\dot H^p(\Omega))$-norm
estimate of the error with an additional log-factor
$\ell_h:=|\ln h|$ for $f\in L^\infty(0,T;\dH q)$, $-1<q\leq1$ (cf. Theorem \ref{thm:gal:linf}):
\begin{equation*}
 \| u_h(t) - u(t) \|_{L^2(\Omega)} + h  \|\nabla(u_h(t) - u(t))\|_{L^2(\Omega)} \le Ch^{2+\min(q,0)} \, \ell_h^{2} \|f\|_{L^\infty(0,t;\dH q)}.
\end{equation*}
Further, we study the more practical lumped mass scheme, and show the same
convergence rates for the gradient.
For right hand side $f\in L^\infty(0,T;\dH q)$, $-1<q<1$, on general
quasi-uniform meshes, we are only able to establish a suboptimal $L^2$-error bound of order $O(h^{1+q}\ell_h^2)$.
For a class of special triangulations satisfying condition \eqref{eqn:condQ}, (almost)
optimal estimates of order $O(h^{2+\min(q,0)})$ and $O(h^{2+\min(q,0)}\ell_h^2)$ hold in $L^2(0,T;L^2(\Omega))$- and
$L^\infty(0,T;L^2(\Omega))$-norm, respectively, cf. Theorems \ref{thm:lump:l2l2}
and \ref{thm:lump:linfl2}. These results extend related results for nonsmooth initial data
obtained in our earlier works \cite{JinLazarovPasciakZhou:2013,Bangti_LZ_2013}.

The rest of the paper is organized as follows. In Section \ref{sec:prelim},
we recall some preliminaries for the convergence analysis, including basic
properties of the Mittag-Leffler function, the smoothing property of \eqref{eq1},
and basic estimates for finite element projection operators. In Sections
\ref{sec:galerkin} and \ref{sec:lump}, we derive error estimates for the standard
Galerkin FEM and lumped mass FEM, respectively. Finally, in Section \ref{sec:numerics}
we present numerical results for both one- and two-dimensional examples, including non-smooth
and very weak data, which confirm our theoretical study. Throughout, we shall denote
by $C$ a generic constant, which may differ at different occurrences and may depend on $T$,
but it is always independent of the solution $u$ and the mesh size $h$.

\section{preliminaries}\label{sec:prelim}
In this part, we recall preliminaries for the convergence analysis, including the Mittag-Leffler function,
solution representation, stability estimates, and basic properties of the finite element projection operators.
\subsection{Mittag-Leffler function}
The Mittag-Leffler function $E_{\alpha,\beta}(z)$ is defined by
\begin{equation*}
  E_{\alpha,\beta}(z) = \sum_{k=0}^\infty \frac{z^k}{\Gamma(k\alpha+\beta)}\quad z\in \mathbb{C}
\quad \mbox{with}
\quad    \Gamma(z) = \int_0^\infty t^{z-1}e^{-t} dt \,\quad ~~\Re(z) >0.
\end{equation*}
The function $E_{\alpha,\beta}(z)$ generalizes the exponential function in that $E_{1,1}(z)= e^z$. The variant $E_{\alpha,\alpha}(z)$,
appears in the kernel of the solution representation of problem \eqref{eq1}; see \eqref{Duhamel} and
\eqref{Duhamel2} below. The following properties are essential in our analysis.
\begin{lemma}
\label{lem:mlfbdd}
Let $0<\alpha<2$ and $\beta\in\mathbb{R}$ be arbitrary, and $\frac{\alpha\pi}{2}
<\mu<\min(\pi,\alpha\pi)$. Then there exists a constant $C=C(\alpha,\beta,\mu)>0$ such
that for $\mu\leq|\mathrm{arg}(z)|\leq \pi$
\begin{equation}\label{M-L-bound}
  |E_{\alpha,\beta}(z)|\leq \left\{\begin{aligned}
    \frac{C}{1+|z|^2}, & \quad \beta-\alpha\in \mathbb{Z}^-\cup\{0\},\\
    \frac{C}{1+|z|},   &  \quad \mbox{otherwise}.
  \end{aligned}\right.
\end{equation}
Moreover, for $\lambda>0$, $\al>0$, and $t>0$ we have
\begin{equation}\label{eq:mlfdiff}
  \begin{aligned}
    \Dal E_{\al,1}(-\la t^\al)&=-\la E_{\al,1}(-\la t^\al) \quad  \mbox{and} \quad
    \frac{d}{dt} E_{\alpha,1}(-\la t^\al)  = -\la t^{\alpha-1}E_{\al.\al}(-\la t^\al).
  \end{aligned}
\end{equation}
\end{lemma}
\begin{proof}
 The estimate \eqref{M-L-bound} can be found in  \cite[pp. 43, equation
(1.8.28)]{KilbasSrivastavaTrujillo:2006} or \cite[Theorem 1.4]{Podlubny_book}, while
 \eqref{eq:mlfdiff} is discussed in \cite[Lemma 2.33, equations
(2.4.58) and (1.10.7)]{KilbasSrivastavaTrujillo:2006}.
\end{proof}

\subsection{Solution representation}\label{ssec:represt}
For the solution representation to problem \eqref{eq1}, we first introduce some notation.
For $s\ge-1$, we denote by $\dH s\subset H^{-1}(\Om)$ the Hilbert space induced by the norm:
\begin{equation*}
  \|v\|_{\dH s}^2=\sum_{j=1}^{\infty}\la_j^s \langle v,\fy_j \rangle^2
\end{equation*}
with $\{\la_j\}_{j=1}^\infty$ and $\{\fy_j\}_{j=1}^\infty$ being respectively the eigenvalues and
the $L^2(\Omega)$-orthonormal eigenfunctions of the Laplace operator $-\Delta$ on the domain
$\Omega$ with a homogeneous Dirichlet boundary condition. As usual, we identify a function $f$
in $L^2(\Omega)$ with the functional $F$ in $H^{-1}(\Omega)$ defined by $\langle F,\phi\rangle =
(f,\phi)$, for all $\phi\in H^1_0(\Omega)$. Then $\{\fy_j\}_{j=1}^\infty$ and $\{\la_j^{1/2}
\fy_j\}_{j=1}^\infty$, form orthonormal basis in $L^2(\Omega)$ and $H^{-1}(\Omega)$, respectively.
Further, $\|v\|_{\dH 0}=\|v\|_{L^2(\Omega)}=(v,v)^{1/2}$ is the norm in $L^2(\Omega)$ and $\|v\|_{\dH {-1}}
= \|v\|_{H^{-1}(\Omega)}$ is the norm in $H^{-1}(\Omega)$. Besides, it is easy to verify that
$\|v\|_{\dH 1}= \|\nabla v\|_{L^2(\Omega)}$ is also the norm in $H_0^1(\Om)$
and  $\|v\|_{\dH 2}=\|\Delta v\|_{L^2(\Omega)}$ is equivalent to the norm in $H^2(\Om)\cap H^1_0(\Om)$
(cf. the proof of Lemma 3.1 of \cite{Thomee97}). Note that $\dot H^s(\Omega)$, $s\ge -1$ form a
Hilbert scale of interpolation spaces. Motivated by this, we denote $\|\cdot\|_{H^s(\Omega)}$ to
be the norm on the interpolation scale between $H^1_0(\Omega)$ and $L^2(\Omega)$ when $s$
is in $[0,1]$ and $\|\cdot\|_{H^{s}(\Omega)}$ to be the norm on the interpolation scale between
$L^2(\Omega)$ and $H^{-1}(\Omega)$ when $s$ is in $[-1,0]$.  Then, $\| \cdot \|_{H^s(\Omega)}$
and $\|\cdot\|_{\dH s}$ are equivalent for $s\in [-1,1]$ by interpolation.

For a Banach space $B$, we define the space
\begin{equation*}
  L^r(0,T;B) = \{u(t)\in B \mbox{ for a.e. } t\in (0,T) \mbox{ and } \|u\|_{L^r(0,T;B)}<\infty\},
\end{equation*}
for any $r\geq 1$, and the norm $\|\cdot\|_{L^r(0,T;B)}$ is defined by
\begin{equation*}
  \|u\|_{L^r(0,T;B)} = \left\{\begin{aligned}\left(\int_0^T\|u(t)\|_B^rdt\right)^{1/r}, &\quad r\in [1,\infty),\\
       \mathrm{esssup}_{t\in(0,T)}\|u(t)\|_B, &\quad r= \infty.
       \end{aligned}\right.
\end{equation*}

Now we give a representation of the solution to problem \eqref{eq1} using
the eigenpairs $\{(\la_j,\fy_j)\}_{j=1}^\infty$.
We define an operator $\bar E(t)$ for $\chi \in L^2(\Omega)$ by
\begin{equation}\label{Duhamel}
{\bar E}(t) \chi = \sum_{j=1}^\infty t^{\al-1} \Mitagaa(-\la_j t^\al)\,(\chi,\fy_j)\, \fy_j(x).
\end{equation}
Then by separation of variables we get the following representation of
the solution $u(x,t)$ to problem \eqref{eq1} for initial data $v=0$:
\begin{equation}\label{Duhamel2}
u(x,t)= \int_0^t  {\bar E}(t-s) f(s) ds. 
\end{equation}

Our first task is to find the weakest class of right hand side data $f(x,t)$ so that \eqref{Duhamel2} is indeed a
solution to the problem \eqref{eq1}. As we see below,
for any $f \in L^2(0,T;\dot H^s(\Omega))$, $-1 < s \le 1$ the function  $u(x,t)$ from
\eqref{Duhamel2} satisfies the differential equation as an element in the space
$L^2(0,T;\dot H^{s+2}(\Omega))$. However, it may not satisfy the homogeneous initial condition
$u(x,0)=0$. In Remark \ref{weakest_sol} we argue that the weakest class of source term that
produces a legitimate weak solution of \eqref{eq1} is $f \in L^r(0,T;\dot H^s(\Omega))$
with $r>1/\al$ and $-1 < s \le 1$. Obviously, for $1/2<\al <1$, the
representation \eqref{Duhamel2} does give a solution $u(x,t) \in L^2(0,T;\dot H^{s+2}(\Omega))$.

We begin with the  following important smoothing property of the solution operator $\bar{E}$.
\begin{lemma}\label{lem:E}
For any $t>0$, we have for $0\le p-q \le 4$
\begin{equation*}
     \|\bar E(t) \chi \|_{\dH p} \le Ct^{-1+\al(1+(q-p)/2)}\|\chi\|_{\dH q}.
\end{equation*}
\end{lemma}
\begin{proof}
The definition of $\bar{E}$ in \eqref{Duhamel} and Lemma \ref{lem:mlfbdd} yield
\begin{equation*}
\begin{split}
\|\bar E(t) \chi \|_{\dH p}^2 & =\sum_{j=1}^{\infty}\la_j^p  |t^{\al-1}\Mitagaa(-\la_j t^\al)|^2 |(\chi,\fy_j)|^2\\
      &=t^{-2+(2+q-p)\al}\sum_{j=1}^{\infty} (\la_j t^{\al})^{p-q}  |\Mitagaa(-\la_j t^\al)|^2 \la_j^q |(\chi,\fy_j)|^2\\
      &\le C t^{-2+(2+q-p)\al}\sum_{j=1}^{\infty} \frac{(\la_j t^{\al})^{p-q}}{(1+(\la_jt^{\al})^2)^2} \la_j^q |(\chi,\fy_j)|^2\\
      &\le C t^{-2+(2+q-p)\al}\sum_{j=1}^{\infty} \la_j^q |(\chi,\fy_j)|^2 \le C t^{-2+(2+q-p)\al}\|\chi\|_{\dH q},
\end{split}
\end{equation*}
where in the last line we have used the inequality
$\sup_{j}\frac{(\la_j t^{\al})^{p-q}}{(1+(\la_jt^{\al})^2)^2} \le C$ for $0\le p-q \le 4$.
From this inequality the desired estimate follows immediately.
\end{proof}

Next we state some stability estimates for the solution $u$ to problem \eqref{eq1} for $f\in
L^\infty(0,T;\dH q)$, $-1<q\leq 1$. These estimates will be essential for the convergence
analysis of the standard Galerkin FEM in Section \ref{sec:galerkin}. The first
estimate in Theorem \ref{thm:reg} in the case $q=0$ was already established in \cite[Theorem 2.1, part (i)]{Sakamoto_2011}.
Below it is extended for the whole range of $q$.
\begin{theorem}\label{thm:reg}
Assume that $f\in L^2(0,T;\dH q)$, $-1<q \le 1$. Then the expression \eqref{Duhamel2} represents a function
$u\in L^2(0,T;\dot H^{q+2}(\Omega))$ which satisfies the differential equation in \eqref{eq1} and the estimate:
\begin{equation}\label{eq:reg}
  \|u\|_{L^2(0,T;\dH {q+2})}+\|\partial_t^\alpha u\|_{L^2(0,T;\dH {q})}\leq C\|f\|_{L^2(0,T;\dH q)}.
\end{equation}
If $f\in L^\infty(0,T;\dH q)$, $-1<q \le 1$, then the function $u(x,t)$ belongs to
$ L^{\infty}(0,T;\dot H^{q+2-\epsilon}(\Omega))$ and satisfies 
\begin{equation}\label{eq:regeps}
    \|u(\cdot,t)\|_{\dH {q+2-\epsilon}} \leq C\epsilon^{-1}t^{\ep\al/2}\|f\|_{L^\infty(0,t;\dH q)}  \quad \mbox{ for any} \quad \epsilon>0  .
\end{equation}
Hence, \eqref{Duhamel2} is a solution to the initial value problem \eqref{eq1}
with a homogeneous initial data $v=0$.
\end{theorem}
\begin{proof}
By the complete monotonicity of the function $E_{\alpha,1}(-t^\alpha)$
(with $\alpha\in (0,1)$) \cite[Lemma 3.3]{Sakamoto_2011}, i.e.,
\begin{equation*}
   (-1)^n\frac{d^n}{dt^n}E_{\alpha,1}(-t^\alpha)\geq 0 \quad\mbox{ for all } t>0,\ \ n=0,1,\ldots,
\end{equation*}
and the differentiation formula in Lemma \ref{lem:mlfbdd}, we deduce
$E_{\alpha,\alpha}(-\eta)\geq0$, $\eta\geq0$. Therefore, for $t>0$
\begin{equation}\label{ineq1}
  \begin{aligned}
    \int_0^t |t^{\alpha-1}E_{\alpha,\alpha}(-\lambda_n t^\alpha)|dt & = \int_0^t t^{\alpha-1} E_{\alpha,\alpha}(-\lambda_n t^\alpha)dt \\
       &  = -\frac{1}{\lambda_n}\int_0^t \frac{d}{dt}E_{\alpha,1}(-\lambda_n t^\alpha)dt\\
       &  = \frac{1}{\lambda_n}(1-E_{\alpha,1}(-\lambda_nt^\alpha)) \le \frac{1}{\lambda_n}.
  \end{aligned}
\end{equation}
Meanwhile, by the differentiation formula \cite[pp. 140-141]{KilbasSrivastavaTrujillo:2006}, we get
\begin{equation*}
  \begin{aligned}
     A_n:=& \partial_t^\alpha \int_0^t(f(\cdot,\tau),\fy_n)(t-\tau)^{\alpha-1}E_{\alpha,\alpha}(-\lambda_n(t-\tau)^\alpha)d\tau\\
     =& (f(\cdot,t),\fy_n) -\lambda_n \int_0^t(f(\cdot,\tau),\fy_n)(t-\tau)^{\alpha-1}E_{\alpha,\alpha}(-\lambda_n(t-\tau)^\alpha)d\tau .
  \end{aligned}
\end{equation*}
Now by means of Young's inequality for convolution, we deduce
\begin{equation*}
  \begin{aligned}
   \|A_n\|_{L^2(0,T)}^2 &\leq C_1\int_0^T|(f(\cdot,t),\fy_n)|^2 dt + C_2\int_0^T |(f(\cdot,t),\fy_n)|^2 dt\left(\int_0^T|\lambda_nt^{\alpha-1}E_{\alpha,\alpha}(-\lambda_nt^\alpha)|dt\right)^2\\
   &\leq C\int_0^T|(f(\cdot,t),\fy_n)|^2 \, dt.
  \end{aligned}
\end{equation*}
Thus there holds
\begin{equation*}
  \begin{aligned}
    \|\partial_t^\alpha u \|_{L^2(0,T;\dH q)}^2 &= \sum_{n=1}^\infty \int_0^T\lambda_n^q|\partial_t^\alpha\int_0^t(f(\cdot,\tau),\fy_n)(t-\tau)^{\alpha-1}E_{\alpha,\alpha}(-\lambda_n(t-\tau)^{\alpha})d\tau|^2 dt \\
    &\leq C\sum_{n=1}^\infty \int_0^T\lambda_n^q|(f(\cdot,t),\phi_n)|^2 dt= C\|f\|_{L^2(0,T;\dH q)}^2.
  \end{aligned}
\end{equation*}
Now using equation \eqref{eq1} and the triangle inequality, we also get $\|\Delta u\|_{L^2(0,T;\dH q)}
\leq C\|f\|_{L^2(0,T;\dot{H}^q(\Omega))}$. This shows the first assertion. By Lemma \ref{lem:E}
we have
\begin{equation*}
  \begin{aligned}
    \|u(\cdot,t)\|_{\dH {q+2-\ep}} & = \|\int_0^t \bar{E}(t-s)f(s)ds\|_{\dH {q+2-\ep}}
      \leq \int_0^t \|\bar{E}(t-s)f(s)\|_{\dH {q+2-\ep}} ds \\
      & \leq C\int_0^t (t-s)^{\ep\al/2-1} \|f(s)\|_{\dH q}ds
      \leq C\epsilon^{-1}t^{\ep\al/2}\|f\|_{L^\infty(0,t;\dot H^q(\Omega))}
  \end{aligned}
\end{equation*}
which shows estimate \eqref{eq:regeps}. Finally, it is follows directly from this
that the representation $u$ satisfies also the initial condition $u(0)=0$, i.e.,
for any $\ep>0$, $\lim_{t\to 0 }\|u(\cdot,t)\|_{\dH {q+2-\ep}}=0$, and thus it is indeed a
solution of the initial value problem \eqref{eq1}.
\end{proof}

\begin{remark}
The first estimate in Theorem \ref{thm:reg} can be improved to
\begin{equation*}
  \|u\|_{L^r(0,T;\dH {q+2})}+\|\partial_t^\alpha u\|_{L^r(0,T;\dot H^{q})}\leq C\|f\|_{L^r(0,T;\dH q)}
\end{equation*}
for any $r\in(1,\infty)$. The proof is essentially identical. The $\epsilon$ factor in the second estimate
reflects the limited smoothing property of the fractional differential operator, resulting
from the slow decay of the Mittag-Leffler function $E_{\alpha,\alpha}(-z)$.
\end{remark}

\begin{remark}\label{weakest_sol}
The condition $f\in  L^\infty(0,T;\dH q)$ can be weakened
to $f\in L^r(0,T;\dH q)$ with $r>1/\alpha$. This follows from Lemma \ref{lem:E} and
the Cauchy-Schwarz inequality with $r'$, $1/r' + 1/r=1$
\begin{equation*}
  \begin{aligned}
    \|u(\cdot,t)\|_{\dH q} & \leq \int_0^t \|\bar{E}(t-s)f(s)\|_{\dH {q}} ds
       \leq \int_0^t (t-s)^{\alpha-1} \|f(s)\|_{\dH q}ds \\
     &  \leq \tfrac{C}{1+r'(\alpha-1)}t^{1+r'(\alpha-1)}\|f\|_{L^r(0,t;\dot H^q(\Omega))},
  \end{aligned}
\end{equation*}
where $1+r'(\alpha-1)>0$ by the condition $r>1/\alpha$. It follows from this that the initial
condition $u(x,0)=0$ holds in the following sense: $\lim_{t\to0^+}\|u(t)\|_{\dH q}=0$.
Hence for any $\alpha\in (1/2,1)$ the representation formula \eqref{Duhamel2} remains a
legitimate solution under the weaker condition $f\in L^2(0,T;\dH q)$.
\end{remark}

\begin{remark}\label{L-infinity}
For the ease of exposition, in the error analysis we shall restrict our discussion
to the case $f\in L^\infty(0,T;\dH q)$. Nonetheless, we note that for $p=0,1$ the
$L^2(0,T;\dH p)$-norm estimate of the error below remain valid under the weakened
regularity condition on the source term $f$.
\end{remark}

\subsection{Ritz and $L^2$-orthogonal  projections}
In our analysis we shall also use the $L^2$-orthogonal projection $P_h:L^2(\Omega)\to X_h$ and
the Ritz projection $R_h:H^1_0(\Omega)\to X_h$, respectively, defined by
\begin{equation*}
  \begin{aligned}
    (P_h \psi,\chi) & =(\psi,\chi) \quad\forall \chi\in X_h,\\
    (\nabla R_h \psi,\nabla\chi) & =(\nabla \psi,\nabla\chi) \quad \forall \chi\in X_h.
  \end{aligned}
\end{equation*}

We shall need some properties of the Ritz projection $R_h$ and the $L^2$-projection $P_h$ \cite{JinLazarovPasciakZhou:2013}.
\begin{lemma}\label{lem:prh-bound}
Let the mesh be quasi-uniform. Then the operator $R_h$ satisfies:
\begin{eqnarray*}
   \|R_h \psi-\psi\|_{L^2(\Omega)}+h\|\nabla(R_h \psi-\psi)\|_{L^2(\Omega)}\le Ch^q\| \psi\|_{\dH q}\for \psi \in \dH q, \ q=1,2.
\end{eqnarray*}
Further, for $s\in [0,1]$ we have
\begin{equation*}
  \begin{aligned}
     \|(I-P_h)\psi \|_{H^s(\Omega)} &\le Ch^{2-s} \|\psi\|_{\dH 2} \for \psi\in
       H^2(\Omega)\cap H^1_0(\Omega),\\
     \|(I-P_h)\psi \|_{H^s(\Omega)} &\le Ch^{1-s} \|\psi\|_{\dH 1} \for \psi\in H^1_0(\Omega).
  \end{aligned}
\end{equation*}
In addition, by duality  $P_h$ is stable on $\dH s$ for $s\in [-1,0]$.
\end{lemma}

\section{Semidiscrete Galerkin FEM}\label{sec:galerkin}

\subsection{Finite element method}\label{ssec:fem}
The semidiscrete Galerkin FEM for problem \eqref{eq1} with $v=0$ is: find $u_h \in X_h$ such that
\begin{equation}\label{fem}
\begin{split}
 {( \Dal u_{h},\chi)}+ a(u_h,\chi)&= {(f, \chi)},
\quad \forall \chi\in X_h,\ T \ge t >0 \quad  \mbox{and} \quad u_h(0)=0.
\end{split}
\end{equation}

Upon introducing the discrete Laplace operator $\De_h: X_h\to X_h$
\begin{equation}\label{eqn:Delh}
  -(\De_h\psi,\chi)=(\nabla\psi,\nabla\chi)\quad\forall\psi,\,\chi\in X_h,
\end{equation}
the spatial discrete problem can be written as
\begin{equation}\label{fem-operator}
   \Dal u_{h}(t)-\De_h u_h(t) =f_h(t) \for t\ge0 \quad \mbox{with} \quad  u_h(0)=0,
\end{equation}
where the discrete source term $f_h=P_h f$.
Then the solution of \eqref{fem-operator} can be represented by the
eigenpairs $\{(\la^h_j,\fy_j^h)\}_{j=1}^{\NN}$ of the discrete Laplacian
$-\De_h$. Now we introduce the discrete analogue $\bar{E}_h $
of the solution operator $\bar{E}$ defined in \eqref{Duhamel} for $t>0 $:
\begin{equation}\label{E-tilde}
  \Etilh(t) f_h= \sum_{j=1}^\NN t^{\al-1} \Mitagaa(-\la^h_jt^\al)\,(f_h,\fy^h_j) \, \fy_j^h.
\end{equation}
Then the solution $u_h(x,t)$ of the discrete problem \eqref{fem-operator} can be expressed by:
\begin{equation}\label{Duhamel_o}
     u_h(x,t)=  
     \int_0^t \Etilh(t-s) f_h(s)\,ds.
\end{equation}

Now we define the discrete norm $\normh{\cdot}{p}$ on $X_h$  for any $p\in\mathbb{R}$
\begin{equation}\label{eqn:normhp}
  \normh{\psi}{p}^2 =
      \sum_{j=1}^N(\la_j^h)^p(\psi,\fy_j^h)^2\quad \psi\in X_h.
\end{equation}
Analogously, we introduce the associated spaces $\tribar \cdot\tribar_{L^r(0,T;\dH p)}$, $r\in[1,\infty]$, on the space $X_h$.

We have the following norm equivalence and inverse inequality.
\begin{lemma}\label{lem:inverse}
For all $\psi \in X_h$ and any real $l>s$
\begin{equation*}
  \normh{ \psi}{l}\le Ch^{s-l}\normh{\psi}{s}.
\end{equation*}
For any $s\in [-1,1]$, the norms $\normh{\cdot}{s}$ and $\|\cdot\|_{\dH s}$ are equivalent on $X_h$.
\end{lemma}
\begin{proof}
The inverse estimate was shown in \cite[Lemma 3.3]{Bangti_LZ_2013}.
By the definition of the $\normh{\cdot}{s}$-norm and the discrete Laplace operator,
it is easy to show $\normh{\cdot}{s}$ is equivalent to $\|\cdot\|_{\dH s}$ with
$s=0,1$. Then the assertion for $0\le s\le 1$ follows by interpolation, and by
duality it is also equivalent to $\|\cdot\|_{\dH s}$ for $-1\leq s\leq0$.
\end{proof}

\subsection{Properties of the discrete solution}\label{ssec:sol}
Next, analogous to Lemma \ref{lem:E}, we introduce some smoothing properties of $\Etilh(t)$.
The proof is identical with that for Lemma \ref{lem:E}, and hence omitted.
\begin{lemma}\label{lem:Eh}
Let $\Etilh$ be defined by \eqref{E-tilde} and $ \psi \in X_h$. Then we have for all $t
>0$,
\begin{equation*}
 \normh{ \Etilh(t) \psi }{p} \le \left \{
\begin{array}{ll}
  Ct^{ -1 + \al(1 + (q-p)/2)}\normh{\psi}{q}, & \quad p-4\leq q \le p, \\[1.3ex]
  Ct^{ -1 + \alpha }\normh{\psi}{q},  & \quad p< q.
\end{array}\right .
\end{equation*}
\end{lemma}

The following estimate is the discrete analogue of Theorem \ref{thm:reg}.
It is essential for the analysis of the lumped mass method in Section 4.
\begin{lemma}\label{lem:reg-d}
Let $u_h$ be the solution of \eqref{fem-operator}. Then for arbitrary $p>-1$
\begin{equation}\label{Gal-L2-est-1}
  \begin{aligned}
   \int_0^T\normh{\Dal u_h(t)}{p}^2 +\normh{ u_h(t) }{p+2}^2 \, dt & \le  \int_0^T \normh {f_h(t)}{p}^2 dt,
  \end{aligned}
\end{equation}
and
\begin{equation}\label{Gal-L2-est-2}
  \begin{aligned}
    \normh{ u_h(t) }{p+2-\epsilon} & \leq C\epsilon^{-1}t^{\ep\al/2}\tribar f_h\tribar_{L^\infty(0,t;\dH {p})}.
  \end{aligned}
\end{equation}
\end{lemma}
\begin{proof}
The solution $u_h(t)$ of \eqref{fem-operator} can be represented by \eqref{Duhamel_o}, and hence
\begin{equation*}
\begin{split}
   \normh {u_h(t)}{p+2}^2
  =&\sum_{j=1}^{N} (\la_j^h)^{p+2} \Big |\int_0^t (t-\tau)^{\al-1}
\Mitagaa(-\la^h_j(t-\tau)^\al)(f_h(\cdot,\tau),\fy_j^h)\,d\tau \Big |^2\\
=&\sum_{j=1}^{N} (\la_j^h)^p \Big |\int_0^t \la_j^h(t-\tau)^{\al-1} \Mitagaa(-\la^h_j(t-\tau)^\al)  (f_h(\cdot,\tau),\fy_j^h)\,d\tau \Big |^2.
\end{split}
\end{equation*}
Then by Young's inequality for convolution, we deduce
\begin{equation*}
 \begin{split}
  \int_0^T\normh{ u_h(t) }{p+2}^2 \, dt
              & \le \sum_{j=1}^N (\la_j^h)^p \int_0^T |(f_h(\cdot,t),\fy_j^h)|^2 \,dt
  \left( \int_0^T \la_j^h t^{\al-1} \Mitagaa(-\la^h_j t^\al) \, dt\right)^2\\
       &\le  C \int_0^T \normh {f_h(t)}{p}^2 dt,
 \end{split}
\end{equation*}
where the last line follows from the identity
$\int_0^T \la_j^h t^{\al-1} \Mitagaa(-\la^h_j t^\al)dt=1-E_{\alpha,1}(-\la^h_jt^\alpha)\in (0,1)$, cf. \eqref{ineq1}.
Now the first estimate follows from this and the triangle inequality and equation \eqref{fem-operator}.

The second estimate follows from Lemma \ref{lem:Eh}
\begin{equation*}
  \begin{aligned}
    \normh {u_h(\cdot,t)}{p+2-\epsilon} & = \normh{\int_0^t \bar{E}_h(t-s)f_h(s)ds}{p} \leq \int_0^t \normh {\bar{E}_h(t-s)f_h(s)}{p+2-\ep} ds \\
      & \leq \int_0^t (t-s)^{\ep\al/2-1} \normh {f_h}{p}
       \leq C\epsilon^{-1}t^{\ep\al/2}\tribar f_h\tribar_{L^\infty(0,t; H^{p}(\Omega))}.
  \end{aligned}
\end{equation*}
This completes the proof of the lemma.
\end{proof}

The rest of this part is devoted to the error analysis for the semi-discrete
Galerkin scheme for a nonsmooth source term
$f\in L^\infty(0,T;\dH q) $, $-1<q\leq 1$.
To this end, we employ the $L^2$-projection $P_hu$, and split the error $u_h-u$ into two terms as:
\begin{equation}\label{separate}
  u_h-u=(u_h-P_hu)+(P_hu-u):=\vtht + \rlh.
\end{equation}
Obviously, $ P_h \Dal \rlh = \Dal P_h(P_hu-u)=0$
and using the identity $\Delta_hR_h=P_h\Delta$, we get the following equation for $\vtht$:
\begin{equation}\label{eq:thettil}
 \Dal \vtht(t) -\Delta_h \vtht(t) =
- \Delta_h (R_h u - P_h u)(t), \quad t>0, \quad \vtht(0)=0.
\end{equation}
Then in view of the representation formula \eqref{Duhamel_o}, $\vtht(t)$ can be represented by
\begin{equation}\label{eqn:vtht}
  \vtht(t) = - \int_0^t\Etilh(t-s)\Delta_h(R_hu-P_hu)(s)\,ds.
\end{equation}
Next we shall treat the $L^2$- and $L^\infty$- in time error estimates separately,
due to the different stability estimates in the two cases, cf. Theorem \ref{thm:reg}.

\subsection{Error estimates for solutions in $L^2(0,T;\dot H^p(\Omega))$}
In this part, we establish error estimates in $L^2$-norm in time.
The case $-1<q\leq 0$ is stated in the next theorem, while
the case $0<q\leq 1$ follows directly and is commented in Remark \ref{rmk:gal:l2} below.
\begin{theorem}\label{thm:gal:l2}
Let $f\in L^\infty(0,T;\dH q)$, $-1<q\leq 0$, and $u$ and $u_h$ be the
solutions of \eqref{eq1} and \eqref{fem-operator} with
$f_h=P_hf$, respectively. Then
\begin{equation*}
 \| u_h - u \|_{L^2(0,T;L^2(\Om))} + h\|\nabla(u_h- u)\|_{L^2(0,T;L^2(\Om))} \le Ch^{2+q}  \|f\|_{L^2(0,T;\dH q)}.
\end{equation*}
\end{theorem}
\begin{proof}
We use the splitting \eqref{separate}.
By Theorem \ref{thm:reg} and Lemma \ref{lem:prh-bound}
\begin{equation*}
\begin{split}
 \| \rlh \|_{L^2(0,T;L^2(\Om))} + h  \|\nabla \rlh\|_{L^2(0,T;L^2(\Om))} &\le Ch^{2+q}  \|u\|_{L^2(0,T;\dH {2+q})}
                   \le Ch^{2+q}  \|f\|_{L^2(0,T;\dH q)}.
\end{split}
\end{equation*}
By \eqref{Duhamel_o}, \eqref{eq:thettil} and Lemmas \ref{lem:reg-d} and \ref{lem:prh-bound}, we have for $p=0,\,1$:
\begin{equation*}
 \begin{split}
   \int_0^T \|\vtht(t)\|^2_{p} dt 
                    & \le C\int_0^T \normh{\Delta_h (R_h u - P_h u)(t) }{p-2}^2 dt \\
                    & \le C\int_0^T \normh{(R_h u - P_h u)(t)}{p}^2 dt\\
                    & \le C h^{4+2q-2p} \| u(t) \|_{L^2(0,T;\dH {2+q})}^2 \le C h^{4+2q-2p}\| f(t) \|_{L^2(0,T;\dH q)}^2.
 \end{split}
\end{equation*}
Combing the preceding two estimates yields the desired assertion.
\end{proof}

\begin{remark}\label{rmk:gal:l2}
Theorem \ref{thm:reg} implies that for $0<q\leq 1$, there holds
\begin{equation*} 
 \| u_h - u \|_{L^2(0,T;L^2(\Om))} + h\|\nabla(u_h- u)\|_{L^2(0,T;L^2(\Om))} \le Ch^2\|f\|_{L^2(0,T;\dH q)}.
\end{equation*}
\end{remark}

\subsection{Error estimates for solutions in $L^\infty(0,T;\dot H^p(\Omega))$} Now we turn to error estimates
in $L^\infty$-norm in time.
Like before, we first consider the case $-1<q\leq 0$. By Theorem \ref{thm:reg}
and Lemma \ref{lem:prh-bound}, the following estimate holds for $\rlh$:
\begin{equation*}
\begin{split}
 \| \rlh(t) \|_{L^2(\Omega)} + h \|\nabla\rlh(t)\|_{L^2(\Omega)} &\le Ch^{2+q-\epsilon} \|u(t)\|_{\dH {2+q-\epsilon}}
          \le C\epsilon^{-1}h^{2+q-\epsilon}\|f\|_{L^\infty(0,t;\dH q)}.
\end{split}
\end{equation*}
Now the choice $\ell_h= |\ln h|,\, \epsilon=1/\ell_h$ yields
\begin{equation}\label{Ph-bound}
 \| \rlh(t) \|_{L^2(\Omega)} + h \|\nabla\rlh(t)\|_{L^2(\Omega)}\le C\ell_hh^{2+q} \|f\|_{L^\infty(0,t;\dH q)}.
\end{equation}

Thus, it suffices to bound the term $\vtht$, which is shown in the following lemma.
\begin{lemma}\label{lem:vtht}
Let $\vtht(t)$ be defined by \eqref{eqn:vtht}. Then for $f\in L^\infty(0,T;\dH q)$, $-1<q\leq 0$, we have
\begin{equation*}
  \|\vtht(t)\|_{L^2(\Omega)}+h\|\nabla\vtht(t)\|_{L^2(\Omega)}\leq Ch^{2+q}\ell_h^2\|f\|_{L^\infty(0,t;\dH q)} \quad \mbox{with } \quad \ell_h=|\ln h|.
\end{equation*}
\end{lemma}
\begin{proof}
By \eqref{Duhamel_o} and Lemmas \ref{lem:inverse} and \ref{lem:Eh}, we deduce that for $p=0,1$
\begin{equation*}
   \begin{aligned}
     \|\vtht(t)\|_{\dH p} &\leq \int_0^t \|\Etilh(t-s)\Delta_h(R_h u-P_hu)(s)\|_{\dH p} ds\\
      &\leq C\int_0^t (t-s)^{\epsilon\al/2 - 1} \normh {\Delta_h(R_h u -P_hu)(s)}{p-2+\epsilon}ds\\
      & \leq C\int_0^t(t-s)^{\epsilon\al/2-1} \normh{R_h u(s) -P_hu(s)}{p+\epsilon}ds := A.
    \end{aligned}
\end{equation*}
Further, we apply the inverse estimate from Lemma \ref{lem:inverse} for $R_h u - P_h u$ and
the bounds in Lemma \ref{lem:prh-bound}, for $P_h u -u$ and $R_hu -u$,
respectively, to deduce
\begin{equation*}
   \begin{aligned}
A & \leq C h^{-\epsilon} \int_0^t(t-s)^{\epsilon\al/2-1} \| R_h u(s) -P_hu(s)\|_{\dot H^p(\Omega)}ds \\
       &  \leq C h^{2+q-p-2\epsilon} \int_0^t(t-s)^{\epsilon\al/2-1} \| u(s) \|_{\dH {2+q-\epsilon}}ds.
   \end{aligned}
\end{equation*}
Further, by applying estimate \eqref{eq:regeps} and choosing $\ep=1/\ell_h$ we get
\begin{equation*}
   \begin{aligned}
    A   & \leq C \epsilon^{-1} h^{2+q-p-2\epsilon}  \|f\|_{L^\infty(0,t;\dH q)}
\int_0^t(t-s)^{\epsilon\al/2-1}t^{\epsilon \alpha/2}ds
\le Ch^{2+q-p}\ell_h^2\|f\|_{L^\infty(0,t;\dH q)}.
   \end{aligned}
\end{equation*}
This completes the proof of the lemma.
\end{proof}
Now we can state the main result of this part, namely, an almost optimal error estimate
 of the Galerkin approximation $u_h$ for solutions $u \in L^\infty(0,T;\dH p)$.
\begin{theorem}\label{thm:gal:linf}
Let $f\in L^\infty(0,T;\dH q)$, $-1<q\leq0$, and $u$ and $u_h$ be the solutions
of \eqref{eq1} and \eqref{fem-operator} with $f_h=P_hf$, respectively.
Then with $\ell_h =| \ln h|$, there holds
\begin{equation*}
 \| u_h(t) - u(t) \|_{L^2(\Om)} +
h  \|\nabla(u_h(t) - u(t))\|_{L^2(\Omega)} \le Ch^{2+q} \ell_h^{2} \|f\|_{L^\infty(0,t;\dH q)}.
\end{equation*}
\end{theorem}

\begin{remark}
In comparison with the $L^2(0,T;\dH p)$-norm estimates, the $L^\infty(0,T;\dH p)$-norm estimates suffer
from the factor $\ell_h^2$. This is due to Lemma \ref{lem:Eh} and
the regularity estimate in Theorem \ref{thm:reg} (in turn this goes back to Lemma \ref{lem:E}),
which is attributed to the limited smoothing property of the time-fractional differential operator.
\end{remark}

\begin{remark}\label{rmk:gallogh}
An inspection of the proof of Lemma \ref{lem:vtht} indicates that for $0<q\leq1$, one can get rid of one factor $\ell_h$, and hence there holds
\begin{equation*}
 \| u_h(t) - u(t) \|_{L^2(\Om)} + h  \|\nabla(u_h(t) - u(t))\|_{L^2(\Om)} \le Ch^2 \ell_h\|f\|_{L^\infty(0,t;\dH q)}.
\end{equation*}
\end{remark}

\subsection{Problems with more general elliptic operators}\label{general_problems}
We can study problem \eqref{fem} with a more general bilinear form
$a(\cdot, \cdot): V \times V\mapsto \mathbb{R}$ of the form:
\begin{equation}\label{general-form}
 a(u,\chi) = \int_\Om (k(x) \nabla u \cdot \nabla \chi + q(x) u \chi)\, dx,
\end{equation}
where $k(x) $ is a symmetric $d \times d$ matrix-valued measurable function on
the domain $\Omega$ with smooth entries and
$q(x)$ is an $L^\infty(\Omega)$-function. We assume  that
\[
  c_0|\xi|^2\leq \xi^T k(x) \xi \leq c_1 |\xi|^2,\ \for \xi \in {\mathbb R}^d,\  x \in \Om,
\]
where $c_0,c_1>0$ are constants, and the bilinear form $a(\cdot, \cdot)$ is coercive 
on $V \equiv H_0^1(\Om)$. Further, we assume that the problem $a(u,\chi)=(f,\chi), \, \forall \chi \in V$
has a unique solution $u\in V$, with the fully regularity $\|u\|_{H^2(\Omega)} \le C\|f\|_{L^2(\Omega)}$.

Under these conditions we can define a positive definite operator
${\mathcal A}:  H^1_0(\Omega) \to H^{-1}(\Omega),$ which has
a complete set of eigenfunctions $\fy_j(x)$ and respective eigenvalues $\la_j({\mathcal A})>0$.
Then we can define the spaces $\dH q$ as in Section \ref{ssec:represt} and all the equivalent properties are satisfied.
Further, we define the discrete operator ${\mathcal A}_h: X_h \to X_h$ by
\begin{equation*}
  ({\mathcal A}_h \psi,\chi) = a(\psi,\chi), ~~\forall \psi, \chi \in X_h .
\end{equation*}
Then all results for problem \eqref{eq1} can be easily
extended to fractional-order problems with
an elliptic part of this more general form.

\section{Lumped mass method}\label{sec:lump}
In this section, we consider the semidiscrete scheme based on the more practical lumped mass FEM  (see, e.g. \cite[Chapter 15, pp. 239--244]{Thomee97}) and derive related error estimates.

\subsection{Lumped mass FEM}
For completeness we shall first introduce this approximation. Let $z_j^\tau$, $j=1,\dots,d+1$ be
the vertices of the $d$-simplex $\tau \in \mathcal{T}_h$. Consider the quadrature
formula
\begin{equation}\label{quadrature}
Q_{\tau,h}(f) = \frac{|\tau|}{d+1} \sum_{j=1}^{d+1} f(z^\tau_j) \approx \int_\tau f dx.
\end{equation}
We then define an approximation of the $L^2(\Om)$-inner product in $X_h$ by
\begin{equation}\label{h-inner}
(w, \chi)_h = \sum_{\tau \in \T_h}  Q_{\tau,h}(w \chi).
\end{equation}

Then the lumped mass Galerkin FEM is: find $ \luh (t)\in X_h$ such that
\begin{equation}\label{fem-lumped}
 {(\Dal \luh, \chi)_h}+ a(\luh,\chi)= (f, \chi)
\quad \forall \chi\in X_h,\ t >0, \quad
\luh(0)=0.
\end{equation}
The lumped mass method leads to a diagonal mass matrix, and thus enhances the
computational efficiency.

We now introduce the discrete Laplace operator $-\bDelh:X_h\rightarrow X_h$, corresponding to
the approximate $L^2(\Omega)$-inner product $(\cdot,\cdot)_h$, defined by
\begin{equation}\label{eqn:bDelh}
  -(\bDelh\psi,\chi)_h = (\nabla \psi,\nabla \chi)\quad \forall\psi,\chi\in X_h.
\end{equation}
Also, we introduce a projection operator $\Pbarh: L^2(\Om) \rightarrow X_h$ by
$$
(\Pbarh f, \chi)_h = (f, \chi), \quad \forall \chi\in X_h.
$$

Similarly as in Section \ref{sec:galerkin},
we introduce the discrete solution operator $\Ftilh$ by
\begin{equation}\label{eqn:Ftilh}
  \Ftilh f_h(t) = \sum_{j=1}^Nt^{\al-1}E_{\al,\al}(-\bar{\la}_j^h t^\al)(f_h,\bar{\fy}_j^h)_h\bar{\fy}_j^h,
\end{equation}
where $\{\bar{\la}_j^h\}_{j=1}^N$ and $\{\bar{\fy}_j^h\}_{j=1}^N$ are respectively the eigenvalues and
the eigenfunctions of $-\bDelh$ with respect to $(\cdot,\cdot)_h$.
Then with $f_h=\Pbarh f$, the solution $\bar{u}_h$ to problem \eqref{fem-lumped} can be
represented by
\begin{equation}\label{eqn:Duhamel_oo}
  \bar{u}_h(t) = 
  \int_0^t\Ftilh(t-s)f_h(s)ds.
\end{equation}

We need the following modification of the discrete norm \eqref{eqn:normhp}, still denoted
by $\normh {\cdot}{p}$, on  the space $X_h$
\begin{equation}\label{eqn:normhbarp}
\normh{\psi}{p}^2 = \sum_{j=1}^N (\bar{\la}_j^h)^p(\psi,\bar{\fy}_j^h)_h^2\quad \forall p\in\mathbb{R}.
\end{equation}
The following norm equivalence result and inverse estimate are useful for our analysis
\begin{lemma}\label{lem:normequivhbarp}
The norm $\normh {\cdot}{p}$
defined in \eqref{eqn:normhbarp} is equivalent to the norm
$\|\cdot\|_{\dH p}$ on the space $X_h$ for $-1\leq p \leq 1$. Further
for all $\psi \in X_h$ we have
for any real $l>s$
\begin{equation}\label{inverse:hbarp}
  \normh{\psi}{l}\le Ch^{s-l} \normh{\psi}{s}. 
\end{equation}
\end{lemma}
\begin{proof}
The norm equivalence for $q=0,1$ is well known.
The interpolation and duality arguments show $\normh{\cdot }{p}$ 
defined in \eqref{eqn:normhbarp} is equivalent to
$\|\cdot\|_{\dH p}$  on the space $X_h$ for $ -1\leq p \le 1$.
\end{proof}

\begin{lemma}\label{lem:Ftilh}
Let $\Ftilh$ be defined by \eqref{eqn:Ftilh}. Then we have for $\psi \in X_h$ and all $t>0$,
\begin{equation*}
  \normh{\Ftilh(t) \psi}{p}\leq \left\{\begin{array}{ll}
      Ct^{-1+\al(1+(q-p)/2)}\normh{ \psi }{q}, & p-4\leq q\leq p,\\[1.3ex]
      Ct^{-1+\al}\normh{ \psi }{q},& p<q.
   \end{array}\right.
\end{equation*}
\end{lemma}

We also need the quadrature error operator $Q_h: X_h\rightarrow X_h$ defined by
\begin{equation}\label{eqn:Q}
  (\nabla Q_h\chi,\nabla \psi) = \epsilon_h(\chi,\psi)
       : = (\chi,\psi)_h-(\chi,\psi)\quad \forall \chi,\psi\in X_h.
\end{equation}
The operator $Q_h$, introduced in \cite{chatzipa-l-thomee12}, represents the quadrature
error (due to mass lumping) in a special way. It satisfies the following error
estimate \cite[Lemma 2.4]{chatzipa-l-thomee12}.
\begin{lemma}\label{lem:Q}
Let $\bDelh$ and $Q_h$ be 
defined by \eqref{eqn:bDelh} and \eqref{eqn:Q},
respectively. Then
\begin{equation*}
  \|\nabla Q_h\chi\|_{L^2(\Omega)}+h\|\bDelh Q_h\chi\|_{L^2(\Omega)} \leq Ch^{p+1}\|\nabla^p\chi\|_{L^2(\Omega)}
\quad \forall \chi\in X_h, ~~~p=0,1.
\end{equation*}
\end{lemma}

The rest of this section is devoted to the error analysis of the lumped mass method for the case $f\in
L^\infty(0,T;\dH q)$, $-1<q\leq1$. The error $\luh(t)-u(t)$ can be split as
$\luh(t)-u(t)=u_h(t) -u(t) + \delta(t)$ with $\delta(t)=\luh(t)-u_h(t)$ and $u_h(t)$
being the standard Galerkin approximation, i.e., the solution of \eqref{fem-operator}. Thus
it suffices to establish proper error bounds for $\delta(t)$.
By the definitions of $u_h(t)$, $\luh(t)$, and $Q_h$ the function $\delta(t)$ satisfies
\begin{equation*}
  \Dal \delta(t) - \bDelh \delta(t) = \bDelh Q_h\Dal u_h(t)
\quad \mbox{ for } \quad T \ge t>0 \quad \mbox{and} \quad \delta(0)=0.
\end{equation*}
By Duhamel's principle \eqref{eqn:Duhamel_oo}, $\delta(t)$ can be expressed as
\begin{equation}\label{splitlump}
  \delta(t) = \int_0^t \Ftilh (t-s)\bDelh Q_h\Dal u_h(s)ds.
\end{equation}
Like before, now we discuss the $L^2$- and $L^\infty$-norm in time error estimates separately.
In the error analysis, the quadrature error operator $Q_h$ and the inverse estimate
play essential role.

\subsection{Error estimates for solutions in $L^2(0,T;\dH p)$}
In this part, we derive an $L^2(0,T;\dH p)$-error estimate, $p=0,1$, for the lumped mass method.
The main result is stated in the following theorem.
\begin{theorem}\label{thm:lump:l2}
Let $f\in L^\infty(0,T;\dH q)$, $-1<q\leq 1$, and $u$ and $u_h$ be the
solutions of \eqref{eq1} and \eqref{fem-lumped} with
$f_h=\bar P_hf$, respectively.  Then there holds
\begin{equation*}
  \begin{aligned}
    \|\nabla (\luh-u)\|_{L^2(0,T;L^2(\Omega))}&\leq Ch^{1+\min(q,0)}\|f\|_{L^2(0,T;\dH q)}, \\
    \|\luh-u\|_{L^2(0,T;L^2(\Omega))}&\leq Ch^{1+q}\|f\|_{L^2(0,T;\dH q)}.\\
  \end{aligned}
\end{equation*}
\end{theorem}
\begin{proof}
By repeating the proof of Lemma \ref{lem:reg-d},  we deduce from \eqref{splitlump}
and Lemmas \ref{lem:normequivhbarp} and \ref{lem:Q} that
\begin{equation*}
 \begin{split}
   \int_0^T \| \nabla \delta(t)\|^2_{L^2(\Omega)} dt
& \le C\int_0^T \normh {\bDelh Q_h\Dal u_h(t)}{-1}^2 dt  \le C\int_0^T \normh{ Q_h\Dal u_h(t) }{1}^2 dt\\
& \le C\int_0^T\| \nabla Q_h \Dal u_h(t)\|_{L^2(\Omega)}^2 dt  \leq Ch^{2} \int_0^T \|\Dal u_h(t)\|_{L^2(\Omega)}^2dt.
 \end{split}
\end{equation*}
The desired assertion for the case $q \ge 0$ now follows immediately from Lemma \ref{lem:reg-d}.

For $-1 < q <0$ we use
the inverse estimate of Lemma \ref{lem:inverse},
the stability of the Galerkin solution $u_h$ established in Lemma \ref{lem:reg-d} and the
stability of $P_h$ from Lemma \ref{lem:prh-bound} to get
\begin{equation*}
  \begin{aligned}
   \int_0^T \|\nabla \delta(t)\|^2_{L^2(\Omega)} dt  
                          & = Ch^{2+2q}\int_0^T\normh{\Dal u_h(t)}{q}^2 dt\\
             & \leq Ch^{2+2q}\int_0^T\normh{f_h(t)}{q}^2 dt \leq Ch^{2+2q}\|f\|_{L^2(0,T;\dH q)}^2.
  \end{aligned}
\end{equation*}
Now we turn to the $L^2$-estimate. By repeating the preceding arguments, we arrive at
\begin{equation*}
 \begin{split}
   \int_0^T \|\delta(t)\|^2_{L^2(\Omega)} dt
                    & \le C\int_0^T \normh {\bDelh Q_h\Dal u_h(t)}{-2}^2 dt  = C\int_0^T \tribar {Q_h\Dal u_h(t)} \tribar_{L^2(\Omega)}^2 dt\\
                    & \le C\int_0^T \|\nabla Q_h\Dal u_h(t) \|_{L^2(\Omega)}^2 dt \le Ch^{4}\int_0^T\|\Dal u_h(t)\|_{\dH 1}^2 dt \\
                    & \leq Ch^{2+2q} \int_0^T \normh{\Dal u_h(t)}{q}^2dt.
 \end{split}
\end{equation*}
where the second line follows from the trivial inequality $\|\chi\|_{L^2(\Omega)}\leq
C\|\nabla \chi\|_{L^2(\Omega)}$ for $\chi\in X_h$ and the norm equivalence in Lemma \ref{lem:normequivhbarp}.
The rest of the proof is identical with that in the preceding part, and hence omitted.
\end{proof}

The estimate in $L^2(0,T;L^2(\Omega))$-norm of Theorem \ref{thm:lump:l2} is suboptimal for any $q<1$.
An optimal estimate can be obtained under an additional condition on the mesh.

\begin{theorem}\label{thm:lump:l2l2}
Let the assumptions in Theorem \ref{thm:lump:l2} be fulfilled and the operator $Q_h$ satisfy
\begin{equation}\label{eqn:condQ}
\|Q_h\chi\|_{L^2(\Omega)}\leq Ch^2\|\chi\|_{L^2(\Omega)}\quad \forall\chi \in X_h.
\end{equation}
Then
\begin{equation*}
  \|\luh-u\|_{L^2(0,T;L^2(\Omega))}\leq Ch^{2+\min(q,0)}\|f\|_{L^2(0,T;\dot H^q(\Omega))}.
\end{equation*}
\end{theorem}
\begin{proof}
It follows from the condition on the operator $Q_h$ that
\begin{equation*}
 \begin{split}
   \int_0^T \|\delta(t)\|^2_{L^2(\Omega)} dt
        & \le C\int_0^T \normh{\bDelh Q_h\Dal u_h}{-2}^2 dt  \le C\int_0^T \tribar Q_h\Dal u_h(t) \tribar^2_{L^2(\Omega)} dt\\
        & \le Ch^4\int_0^T \|\Dal u_h(t) \|_{L^2(\Omega)}^2 dt \le Ch^{4+2\min(q,0)}\int_0^T\normh{ \Dal u_h(t)}{q}^2 dt.
 \end{split}
\end{equation*}
The rest of the proof is identical with that of Theorem \ref{thm:lump:l2}, and this completes the proof.
\end{proof}

\begin{remark}
The condition \eqref{eqn:condQ} on the operator $Q_h$ is satisfied
for symmetric meshes \cite[section 5]{chatzipa-l-thomee12}. In one dimension, the symmetry requirement can
be relaxed to almost symmetry \cite[section 6]{chatzipa-l-thomee12}. In case \eqref{eqn:condQ} does not hold, we were
able to show only a suboptimal $O(h^{1+q})$-convergence rate for the $L^2$-norm of the error, which
is reminiscent of that for the initial value problem for the classical parabolic equation
\cite[Theorem 4.4]{chatzipa-l-thomee12} and time-fractional diffusion problem \cite[Theorem 4.5]{Bangti_LZ_2013}.
\end{remark}

\subsection{Error estimates for solutions in $L^\infty(0,T;\dH p)$}
Now we derive an estimate in $L^\infty(0,T;\dH p)$-norm for the lumped mass approximation $\luh$.
\begin{theorem}\label{thm:lump:linf}
Let $f\in L^\infty(0,T;\dH q)$, $-1<q\leq 1$, and $u$ and $\luh$ be the solutions of
\eqref{eq1} and \eqref{fem-lumped}, respectively, with $\bar f_h=\bar P_hf$. Then
with $\ell_h=|\ln h|$, the following estimates are valid for $t >0$:
\begin{equation}\label{H1-lumpled}
   \begin{aligned}
     \|\nabla(\luh(t)-u(t))\|_{L^2(\Omega)}&\leq Ch^{1+q} 
               \ell_h^2\|f\|_{L^\infty(0,t;\dH q)}
\quad \mbox{for} \quad -1< q \le 0.
   \end{aligned}
\end{equation}
Moreover, for $-1< q \le 1$ we have
\begin{equation}\label{L2-lumpled}
   \begin{aligned}
     \|\luh(t)-u(t)\|_{L^2(\Omega)} &\leq Ch^{1+q}\ell_h^2 \|f\|_{L^\infty(0,t; \dH q)}.
   \end{aligned}
\end{equation}
\end{theorem}
\begin{proof}
We first note that by the smoothing property of $\Ftilh$ in Lemma \ref{lem:Ftilh} and the inverse
inequality \eqref{inverse:hbarp}, we have for $\chi\in X_h $, $\epsilon>0$, and $p=0,1$
\begin{equation}\label{eqn:lumpbasic1}
  \begin{aligned}
    \normh{ \Ftilh(t)\bDelh Q_h\chi}{p} &\leq Ct^{\epsilon\al/2-1}
                   \normh{ \bDelh Q_h\chi}{p-2+\epsilon}
         = Ct^{\epsilon\al/2-1} \normh{ Q_h\chi}{p+\epsilon}\\
        & \leq Ct^{\epsilon\al/2-1}h^{-\epsilon}\normh{ Q_h \chi}{p}.
  \end{aligned}
\end{equation}

We first prove estimate \eqref{H1-lumpled}. Setting $\chi=\Dal u_h(t)$ in \eqref{eqn:lumpbasic1}
and Lemma \ref{lem:Q} yield
\begin{equation}\label{eqn:lumpbasicI1}
  \begin{aligned}
    \normh{ \Ftilh(t-s)\bDelh Q_h\Dal u_h(s)}{1} &\leq C(t-s)^{\epsilon\al/2-1}h^{-\epsilon} \|Q_h\Dal u_h(s)\|_{\dH 1}\\
        & \le Ch^{1-\ep}(t-s)^{\ep\al/2-1}\|\Dal u_h(s)\|_{L^2(\Omega)}.
  \end{aligned}
\end{equation}
Then it follows from relation \eqref{fem} of Galerkin
approximation $u_h$ and the triangle and inverse inequalities that
\begin{equation*}
  \begin{aligned}
     \|\nabla\delta(t)\|_{L^2(\Omega)} &\leq Ch^{1-\epsilon} \int_0^t (t-s)^{\ep\al/2-1}\|\Dal u_h(s)\|_{L^2(\Omega)}ds \\
     & \leq Ch^{1-\epsilon} \int_0^t (t-s)^{\ep\al/2-1}(\|\Delta_h u_h(s)\|_{L^2(\Omega)}+\|f_h(s)\|_{L^2(\Omega)})ds \\
     & \leq Ch^{1-\epsilon} \int_0^t (t-s)^{\ep\al/2-1}(h^{-\ep}\normh{\Delta_h u_h(s)}{-\ep}+\|f_h(s)\|_{L^2(\Omega)})ds.
  \end{aligned}
\end{equation*}
Further, using the discrete stability estimate in Lemma \ref{lem:reg-d} and
the estimate of $P_h$ in Lemma \ref{lem:prh-bound} we further get for $-1 <q \le 0$
\begin{equation*}
  \begin{aligned}
  \|\nabla\delta(t)\|_{L^2(\Omega)}     & \leq Ch^{1+q-\epsilon} \int_0^t (t-s)^{\ep\al/2-1}(h^{-\ep}\normh{\Delta_h u_h(s)}{q-\ep}+\normh{ f_h(s)}{q})ds \\
     & \leq Ch^{1+q-2\epsilon} \int_0^t(t-s)^{\ep\al/2-1}(\ep^{-1}s^{\ep\al/2}\tribar f_h\tribar_{L^\infty(0,s;\dH q)}
      +\normh{f_h(s)}{q})ds\\
     & \leq C\epsilon^{-1}h^{1+q -2\epsilon} \tribar f_h\tribar_{L^\infty(0,t;\dH q)}\int_0^t(t-s)^{\ep\al/2-1}s^{\ep\al/2} ds \\
     & \leq C\epsilon^{-2}h^{1+q -2\epsilon} \tribar f_h\tribar_{L^\infty(0,t;\dH q)} \leq Ch^{1+q} \ell_h^2 \|f\|_{L^\infty(0,t;\dH q)},
  \end{aligned}
\end{equation*}
where in the last inequality we have chosen $\epsilon=1/\ell_h$. Now \eqref{H1-lumpled}
follows from this and Theorem \ref{thm:gal:linf}.

Next we derive the $L^2$-error estimate. Similar to the derivation of \eqref{eqn:lumpbasic1}, we get
\begin{equation*}
   \begin{aligned}
       \| \Ftilh(t-s)\bDelh Q_h\Dal u_h(s)\|_{L^2(\Omega)} &\leq C (t-s)^{\al/2-1} \|\nabla Q_h\Dal u_h(s)\|_{L^2(\Omega)}\\
        &\le Ch^{2}(t-s)^{\al/2-1}\| \Dal u_h(s) \|_{\dH 1}.
   \end{aligned}
\end{equation*}
Consequently, by the triangle inequality, inverse estimate in Lemma \ref{lem:inverse} and the
discrete stability estimate in Lemma \ref{lem:reg-d}, there holds
\begin{equation*}
   \begin{aligned}
      \|\delta(t)\|_{L^2(\Omega)} &\leq Ch^{2}\int_0^t(t-s)^{\al/2-1}(\|\Delta u_h(s)\|_{\dH 1}+\|P_hf(s)\|_{\dH 1}) ds \\
      &\leq Ch^{2}\int_0^t(t-s)^{\al/2-1}(h^{-\ep}\normh{\Delta u_h(s)}{1-\ep}+\normh{ f_h(s)}{1}) ds \\
      &\leq Ch^{1+q}\int_0^t(t-s)^{\al/2-1}(h^{-\ep}\normh{\Delta u_h(s)}{q-\ep}+\normh{ f_h(s)}{q}) d\tau \\
      &\leq Ch^{1+q}\int_0^t(t-s)^{\al/2-1}(\ep^{-1}h^{-\ep}s^{\ep\al/2}\tribar f_h\tribar_{L^\infty(0,s;\dH q)}+\normh{f_h(s)}{q})ds.
   \end{aligned}
\end{equation*}
The $L^2$-estimate follows directly from Theorem \ref{thm:gal:linf} and
setting $\epsilon=1/\ell_h$ in the above inequality.
\end{proof}

\begin{remark}
For $q>0$, we have $\|\nabla(\luh(t)-u(t))\|_{L^2(\Omega)}\leq Ch\ell_h^2$
and the error cannot be improved even if the function $f$ is smoother.
In view of Remark \ref{rmk:gallogh}, for $0<q\leq1$, the following slightly improved estimate holds
\begin{equation*}
 \|\luh(t)-u(t)\|_{L^2(\Omega)}\leq Ch^{1+q}\ell_h \|f\|_{L^\infty(0,t; \dH q)}.
\end{equation*}
\end{remark}

In the case of $f\in L^\infty(0,T;\dH q)$, $0<q\leq1$, we can obtain an improved
estimate of $\|\nabla \delta(t)\|_{L^2(\Omega)}$:
\begin{equation*}
  \begin{aligned}
     \|\nabla\delta(t)\|_{L^2(\Omega)} &\leq Ch^{2-\epsilon} \int_0^t (t-s)^{\ep\al/2-1}\|\Dal u_h(s)\|_{\dH 1}ds \\
     & \leq Ch^{1+q-\epsilon} \int_0^t (t-s)^{\ep\al/2-1}(h^{-\ep}\normh{\Delta_h u_h(s)}{q-\ep}+\tribar f_h(s)\tribar_{\dH q})ds \\
     & \leq Ch^{1+q-2\epsilon} \int_0^t (t-s)^{\ep\al/2-1}(\ep^{-1}s^{\ep\al/2}\tribar f_h\tribar_{L^\infty(0,s;\dH q)}+\tribar f_h(s)\tribar_{\dH q})ds\\
     & \leq C\epsilon^{-2}h^{1+q-2\epsilon} \|f\|_{L^\infty(0,t;\dH q)} \leq Ch^{1+q} \ell_h^2 \|f\|_{L^\infty(0,t;\dH q)}.
  \end{aligned}
\end{equation*}

  We record this observation in a remark.
\begin{remark}
In the case of a right hand side of intermediate smoothness, i.e., $f\in L^\infty(0,T;\dH q)$, $0< q\leq1$, the gradient
estimate $\|\nabla\delta(t)\|_{L^2(\Omega)}$ can be improved to $(1+q)$th-order at the expense
of an extra factor $\ell_h$:
\begin{equation*}
     \|\nabla\delta(t)\|_{L^2(\Omega)} \leq Ch^{1+q} \ell_h^2 \|f\|_{L^\infty(0,t;\dH q)}.
\end{equation*}
\end{remark}

Like in the case of $L^2(0,T;L^2(\Omega))$-estimate, the $L^\infty(0,T;L^2(\Omega))$ estimate
is suboptimal for any $q\in(-1,1)$, and can
be improved to an almost optimal one by imposing condition \eqref{eqn:condQ}.
\begin{theorem}\label{thm:lump:linfl2}
Let the conditions in Theorem \ref{thm:lump:linf} be fulfilled and the operator $Q_h$ satisfy \eqref{eqn:condQ}, i.e.,
$\|Q_h\chi\|_{L^2(\Omega)}\leq Ch^2\|\chi\|_{L^2(\Omega)}$ for all $\chi \in X_h$.
Then there holds  (with $\ell_h=|\ln h|$)
\begin{equation*}
   \|\luh(t)-u(t)\|_{L^2(\Omega)}\leq Ch^{2+\min(q,0)} \ell_h^2 \|f\|_{L^\infty(0,t;\dH q)}, \quad -1 < q \le 1.
\end{equation*}
\end{theorem}
\begin{proof}
If \eqref{eqn:condQ} holds, then applying \eqref{eqn:lumpbasic1} with $p=0$ we get
\begin{equation*}
  \begin{aligned}
    \|\Ftilh(t-s)\bDelh Q_h\Dal u_h(s)\|_{L^2(\Omega)} &\leq C(t-s)^{\ep\al/2-1}h^{-\epsilon} \|Q_h\Dal u_h(s)\|_{L^2(\Omega)}\\
        & \le Ch^{2-\ep}(t-s)^{\ep\al/2-1}\|\Dal u_h(s)\|_{L^2(\Omega)}.
  \end{aligned}
\end{equation*}
Consequently, this together with the inverse estimate from Lemma \ref{lem:inverse},
the discrete stability estimate in Lemma \ref{lem:reg-d}, and
the stability of $P_h$ in Lemma \ref{lem:prh-bound}, yields
\begin{equation*}
  \begin{aligned}
     \|\delta(t)\|_{L^2(\Omega)} & \leq Ch^{2-\epsilon} \int_0^t (t-s)^{\ep\al/2-1}(\|\Delta_h u_h(s)\|_{L^2(\Omega)}+\|f_h(s)\|_{L^2(\Omega)})ds\\
     & \leq Ch^{2+\min(q,0)-\epsilon} \int_0^t (t-s)^{\ep\al/2-1}(h^{-\ep}\normh{\Delta_h u_h(s)}{-\ep+q}+\normh{ f_h(s)}{q})ds \\
     & \leq Ch^{2+\min(q,0)-\epsilon} \int_0^t (t-s)^{\ep\al/2-1}(h^{-\ep}\ep^{-1}s^{\ep\al/2}\tribar f_h\tribar_{L^\infty(0,s;\dH q)}+\normh {f_h(s)}{q})ds\\
     & \leq C\epsilon^{-2}h^{2+\min(q,0)-2\epsilon} \|f\|_{L^\infty(0,t;\dH q)} \leq Ch^{2+\min(q,0)} \ell_h^2 \|f\|_{L^\infty(0,t;\dH q)},
  \end{aligned}
\end{equation*}
where the last line follows from the choice $\epsilon = 1/\ell_h$. Now the desired assertion
follows immediately from this and Theorem \ref{thm:gal:linf}.
\end{proof}

\section{Numerical results}\label{sec:numerics}

In this section, we present numerical results to verify the theoretical error estimates in
Sections \ref{sec:galerkin} and \ref{sec:lump}. We
present the errors $\|u-\luh\|_{L^\infty(0,T;L^2(\Omega))}$ and $\|\nabla(u-\luh)\|_{L^\infty(0,T;
L^2(\Omega))}$ for the lumped mass method only, since the errors for the Galerkin FEM are almost
identical. In the tables and figures below, we use the following notation convention:  $\|u(t)-\luh(t)
\|_{\dH p}$ for $p=0$ and $p=1$ is simply referred to as $L^2$-norm and $H^1$-norm error, respectively.

\subsection{1-d examples}\label{ssect:1D}
First, we consider \eqref{eq1} for $d=1$ on the unit interval $\Omega=(0,1)$ and
perform numerical tests on the following three data sets:
\begin{enumerate}
 \item[(1a)]  Nonsmooth data: $f(x,t)=(\chi_{[1/2,1]}(t)+1)\chi_{[0,{1/2}]}(x)$,
 where $\chi_S$ is
the characteristic function of the set $S$. The jump at $x=1/2$ leads to
$f(t,\cdot) \notin \dot H^1(\Omega)$; nonetheless,
for any $\epsilon >0$, $f(x,t) \in L^\infty(0,T;H^{{1/2}-\epsilon}(\Omega))$.
\item[(1b)] Very weak data: $f(x,t)=(\chi_{[1/2,1]}(t)+1)\delta_{1/2}(x)$
where $f(x,t)$ involves a Dirac $\delta_{1/2}(x)$-function
concentrated at $x=0.5$.
\item[(1c)] Variable coefficients: in \eqref{general-form}
 we take $k(x)= 3 +\sin(2\pi x)$, $f(x,t)=(\chi_{[1/2,1]}(t)+1)\chi_{[0,1/2]}(x)$,
and $q(x)=0$.
\end{enumerate}

The exact solution for
examples (1a) and (1b) can be explicitly expressed by an infinite series involving
the Mittag-Leffler function $E_{\alpha,\alpha}(z)$ as \eqref{Duhamel2}.
To accurately evaluate the Mittag-Leffler functions,
we employ the algorithm developed in \cite{Seybold:2008}, which is
based on three different approximations of the function, i.e., Taylor series, integral
representations and exponential asymptotics, in different regions of the domain.

In our computation, we divide the unit interval $(0,1)$ into $N+1$ equally spaced subintervals, with a mesh size $h=1/(N+1)$.
The finite element space $X_h$ consists of continuous piecewise linear functions. 
The eigenpairs $(\la^h_j, \fy^h_j(x) )$ and $( \bar\la^h_j, \bar \fy^h_j(x) )$ of the one-dimensional
discrete Laplacian $-\Delta_h$ and $-\bar \Delta_h$, defined by \eqref{eqn:Delh} and \eqref{eqn:bDelh},
respectively, satisfy
\begin{equation*}
  (-\Delta_h \fy^h_j,v)=\la^h_j (\fy^h_j,v) \quad\mbox{ and }\quad
  (-\bar\Delta_h \bar\fy^h_j,v)_h=\bar\la^h_j (\bar\fy^h_j,v)_h \quad \forall v \in X_h.
\end{equation*}
Here $(w,v)$ and $(w,v)_h$ refer to the standard $L^2$-inner product and
the approximate $L^2$-inner product (cf. eq. \eqref{h-inner})
on the space $X_h$, respectively.
Then for $ j=1,\dots,N$, there hold
\begin{equation*}
 \la^h_j= \bar{\la}_j^h /(1 - \tfrac{h^2}{6} \bar{\la}_j^h), ~~ 
  \bar{\la}_j^h=\frac{4}{h^2}\sin^2\frac{\pi j}{2(N+1)},
 ~~\mbox{and} ~~ \fy^h_j(x_k)=\bar\fy^h_j(x_k) = \sqrt2 \sin(j\pi x_k)
\end{equation*}
for $x_k$ being a mesh
point and $ \fy^h_j$ and $ \bar\fy^h_j$
linear over the finite elements. Then the solutions of the standard Galerkin method and lumped mass method can
be computed by \eqref{Duhamel_o} and \eqref{eqn:Duhamel_oo}, respectively.

In the case of (1c), there is no convenient solution representation. To compute the semidiscrete solution,
we have used a direct numerical technique
by first discretizing the time interval, $t_n=n\tau$,
$n=0,1,\dots$, with $\tau$ being the time step size, and then using a weighted
finite difference approximation of the fractional derivative $\Dal u(x,t_n)$
developed in \cite{LinXu:2007}:
\begin{equation*}\label{eqn:timedisc}
  \begin{aligned}
    \Dal u(x,t_n) &= \frac{1}{\Gamma(1-\al)}\sum^{n-1}_{j=0}\int^{t_{j+1}}_{t_j}
       \frac{\partial u(x,s)}{\partial s} (t_n-s)^{-\al}\, ds
    \approx \frac{1}{\Gamma(2-\alpha)}\sum_{j=0}^{n-1}b_j\frac{u(x,t_{n-j})-u(x,t_{n-j-1})}{\tau^\alpha},
  \end{aligned}
\end{equation*}
where the weights $\{b_j\}$ are given by $b_j=(j+1)^{1-\alpha}-j^{1-\alpha}$, $j=0,1,\ldots,n-1$.
If the solution $u(x,t)$ is sufficiently smooth, then the
local truncation error of the finite difference approximation is bounded
by $C\tau^{2-\al}$ for some $C$ depending only
on $u$ \cite[equation (3.3)]{LinXu:2007}. Hence with a small $\tau$,
the fully discrete solution can approximate the semidiscrete solution well.

\subsubsection{Numerical results for example (1a)}
In Tables \ref{tab:nonsmooth1DL2} and \ref{tab:nonsmooth1DLinf} we show
the errors $\|u-\luh\|_{L^2(0,T;\dH p)}$ and $\|u-\luh\|_{L^\infty(0,T;\dH p)}$ for $p=0,1$.
The convergence rates are almost identical for the $L^2$- and $L^\infty$- in time estimates since
both the $L^2(\Omega)$-norm of the error $e=u-\luh$ and its gradient
are bounded independent of the time. Thus in what follows, we only present
the errors $\|u(t)-\luh(t)\|_{L^2(\Omega)}$ and $\|\nabla(u(t)-\luh(t))\|_{L^2(\Omega)}$.
In Figure \ref{fig:nonsmooth_Linf}, we show the plot of the results from
Table \ref{tab:nonsmooth1DLinf} in a log-log scale. The errors are almost
independent of the fractional order $\alpha$, and hence the three curves nearly coincide with each other.
The numerical results fully confirm our theoretical predictions, i.e., $O(h^2)$ and $O(h)$
convergence rates for the $L^2(\Omega)$- and $H^1(\Omega)$-norms of the error, respectively.

\begin{table}[!ht]
\caption{Numerical results, i.e., errors $\|u-\luh\|_{L^2(0,1;L^2(\Omega))}$ ($L^2$-norm) and
$\|u-\luh\|_{L^2(0,1;\dot H^1(\Omega))}$ ($H^1$-norm), for example (1a) (nonsmooth data)
with mesh sizes $h=1/2^k$.}
\label{tab:nonsmooth1DL2}
\begin{center}
     \begin{tabular}{|c|c|c|c|c|c|c|c|}
     \hline
      $\al$ & $k$ & $3$ & $4$ & $5$ &$6$ & $7$ & rate\\
     \hline
     $\al=0.1$ & $L^2$-norm & 9.96e-4 &2.49e-4  &6.23e-5  &1.55e-5  &3.89e-6  & $O(h^{2.00})$ \\
     \cline{2-8}
     & $H^1$-norm           & 2.45e-2 & 1.23e-2 & 6.13e-3 & 3.07e-3 & 1.53e-3 & $O(h^{1.00})$  \\
     \hline
     $\al=0.5$ & $L^2$-norm  & 9.97e-4 &2.50e-4  &6.24e-5  &1.56e-5  &3.90e-6 & $O(h^{2.00})$  \\
     \cline{2-8}
     & $H^1$-norm          & 2.46e-2 & 1.23e-2 & 6.14e-3 & 3.08e-3 & 1.54e-3  & $O(h^{1.00})$   \\
     \hline
     $\al=0.95$ & $L^2$-norm   & 1.01e-3 &2.51e-4  &6.28e-5  &1.57e-5  &3.93e-6 & $O(h^{2.00})$\\
     \cline{2-8}
     & $H^1$-norm           & 2.48e-2 & 1.24e-2 & 6.20e-3 & 3.10e-3 & 1.55e-3  & $O(h^{1.00})$\\
     \hline
     \end{tabular}
\end{center}
\end{table}

\begin{table}[!ht]
\caption{Numerical results, i.e., errors $\| u(t)-\luh(t) \|_{\dH p}$, $p=0,1$ at $t=1$, for
example (1a) (nonsmooth data) with mesh sizes $h=1/2^k$.}
\label{tab:nonsmooth1DLinf}
\begin{center}
     \begin{tabular}{|c|c|c|c|c|c|c|c|}
     \hline
      $\al$ & $k$ & $3$ & $4$ & $5$ &$6$ & $7$ & rate\\
     \hline
     $\al=0.1$ & $L^2$-norm & 9.98e-4 &2.50e-4  &6.24e-5  &1.56e-5  &3.90e-6  & $O(h^{2.00})$ \\
     \cline{2-8}
     & $H^1$-norm           & 2.47e-2 & 1.23e-2 & 6.17e-3 & 3.09e-3 & 1.54e-3 & $O(h^{1.00})$  \\
     \hline
     $\al=0.5$ & $L^2$-norm  & 1.01e-3 &2.51e-4  &6.29e-5  &1.57e-5  &3.93e-6 & $O(h^{2.00})$  \\
     \cline{2-8}
     & $H^1$-norm          & 2.53e-2 & 1.27e-2 & 6.33e-3 & 3.17e-3 & 1.58e-3  & $O(h^{1.00})$   \\
     \hline
     $\al=0.95$ & $L^2$-norm   & 1.04e-3 &2.59e-4  &6.49e-5  &1.62e-5  &4.06e-6 & $O(h^{2.00})$\\
     \cline{2-8}
     & $H^1$-norm           & 2.58e-2 & 1.29e-2 & 6.46e-3 & 3.23e-3 & 1.62e-3  & $O(h^{1.00})$\\
     \hline
     \end{tabular}
\end{center}
\end{table}

\begin{figure}[h!]
\center
  \includegraphics[trim = 1cm .1cm 2cm 0.5cm, clip=true,width=10cm]{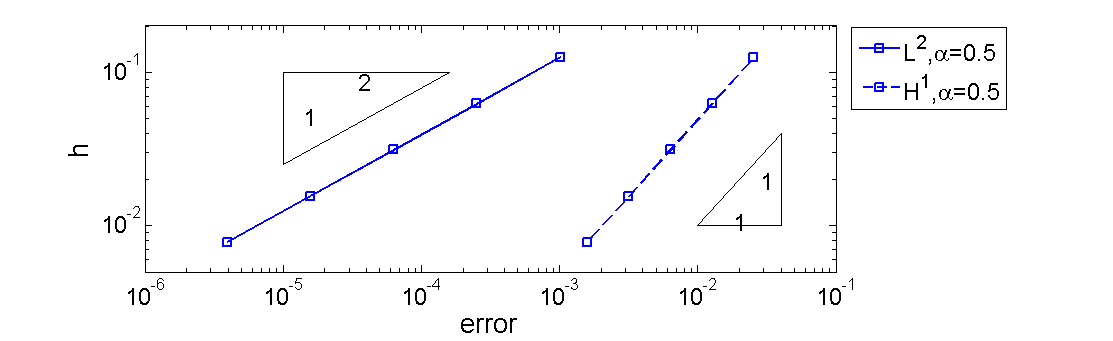}
  \caption{Numerical results, i.e., errors $\|u(t)-\luh(t)\|_{\dH p}$, $p=0,1$ at $t=1$,
  for example (1a) (nonsmooth data) with $\alpha=0.5$.}\label{fig:nonsmooth_Linf}
\end{figure}

\subsubsection{Numerical results for example (1b)}
In this example, we consider the case of very weak data, which involves the
Dirac $\delta$-function. Since the Dirac $\delta$-function $\delta_{1/2}(x)$
is a bounded linear functional on $C_0(0,1)$.
and $H_0^{1/2+\epsilon}(0,1)$ with $\epsilon>0$ embeds continuously
into $C_0(0,1)$, we have $\delta_{1/2}(x) \in H^{-1/2-\epsilon}(\Omega)$.
Thus, the source term $f(x,t)=(\chi_{[1/2,1]}(t)+1)\delta_{1/2}(x) \in
L^{\infty}(0,T;H^{-1/2-\epsilon}(\Omega))$. In Table \ref{tab:weak1D} we show the
error and convergence rates for the lumped mass method for three different $\al$ values.
Here the mesh size is chosen to be $h=1/(2^k+1)$, and thus the Dirac $\delta$ function
is not aligned with the grid. The results indicate an $O(h^{1/2})$ and $O(h^{3/2})$ convergence
rate for the $H^1(\Omega)$- and $L^2(\Omega)$-norm of the error, respectively, which agrees well with the
theoretical prediction. In Table \ref{tab:weak1Dsuper}, we report the results for the case
that the $\delta$-function is supported at a grid point. We observe that the method converges
at the expected rate in $H^1(\Omega)$-norm, while the convergence rate in the $L^2(\Omega)$-norm is $O(h^2)$,
i.e., the method exhibits a superconvergence of one half order, which theoretically remains to be
established. The plots of the results for $\alpha=0.5$ in Tables \ref{tab:weak1D} and
\ref{tab:weak1Dsuper} are shown respectively in Figures \ref{fig:weak1D} and \ref{fig:weak1Dsuper}
in a log-log scale.

\begin{table}[!ht]
\caption{Numerical results, i.e., errors $\|u(t)-\luh(t) \|_{\dH p}$, $p=0,1$ at $t=1$, for
example (1b) (Dirac $\delta$-function) with mesh sizes $h=1/(2^k+1)$.}
\label{tab:weak1D}
\begin{center}
     \begin{tabular}{|c|c|c|c|c|c|c|c|}
     \hline
      $\al$ & $k$ & $3$ & $4$ & $5$ &$6$ & $7$ & rate\\
     \hline
     $\al=0.1$ & $L^2$-norm & 5.35e-3 & 2.07e-3  & 7.67e-4 &  2.78e-4 &  9.94e-5  & $O(h^{1.44})$ \\
     \cline{2-8}
     & $H^1$-norm           & 1.56e-1   &1.14e-1   &8.23e-2  &5.88e-2   &4.17e-2 & $O(h^{0.48})$  \\
     \hline
     $\al=0.5$ & $L^2$-norm  & 5.38e-3 & 2.08e-3  & 7.68e-4 &  2.78e-4 &  9.95e-5 & $O(h^{1.44})$  \\
     \cline{2-8}
     & $H^1$-norm           & 1.57e-1   &1.15e-1   &8.25e-2  &5.89e-2   &4.17e-2 & $O(h^{0.48})$   \\
     \hline
     $\al=0.95$ & $L^2$-norm   & 5.39e-3 & 2.08e-3  & 7.69e-4 &  2.79e-4 &  9.95e-5 & $O(h^{1.44})$\\
     \cline{2-8}
     & $H^1$-norm           & 1.58e-1   &1.15e-1   &8.26e-2  &5.89e-2   &4.18e-2& $O(h^{0.48})$\\
     \hline
     \end{tabular}
\end{center}
\end{table}
\begin{figure}[h!]
\center
  \includegraphics[trim = 1cm .1cm 2cm 0.5cm, clip=true,width=10cm]{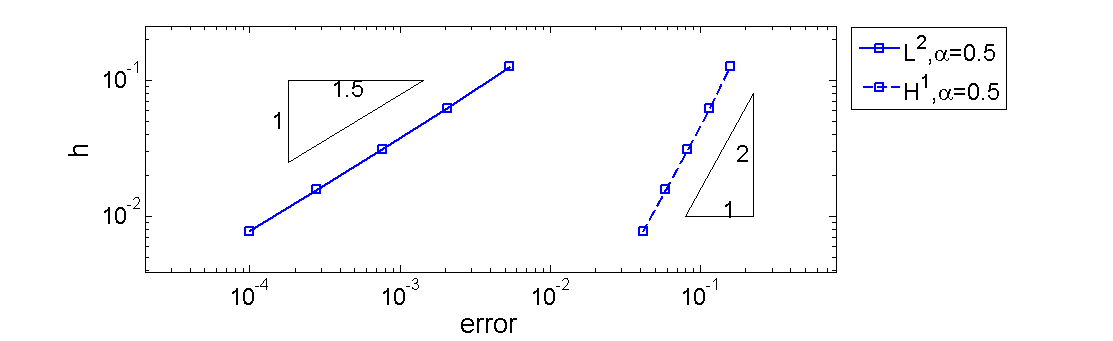}
  \caption{Numerical results, i.e., errors $\|u(t)-\luh(t)\|_{\dH p}$, $p=0,1$ at $t=1$, for example (1b) (Dirac
  $\delta$-function) with $\alpha=0.5$. The Dirac $\delta$-function is not aligned with grid points.}\label{fig:weak1D}
\end{figure}

\begin{table}[!ht]
\caption{Numerical results, errors $\|u(t)-\luh(t) \|_{\dH p}$, $p=0,1$ at $t=1$,
for example (1b) (Dirac $\delta$ function) with mesh sizes $h=1/2^k$.}
\label{tab:weak1Dsuper}
\begin{center}
     \begin{tabular}{|c|c|c|c|c|c|c|c|}
     \hline
      $\al$ & $k$ & $3$ & $4$ & $5$ &$6$ & $7$ & rate\\
     \hline
     $\al=0.1$ & $L^2$-norm & 1.59e-4 &3.93e-5  &9.80e-6  &2.45e-6  &6.15e-7  & $O(h^{2.00})$ \\
     \cline{2-8}
     & $H^1$-norm           & 5.46e-2 & 3.91e-2 & 2.78e-2 & 1.97e-3 & 1.39e-3 & $O(h^{0.49})$  \\
     \hline
     $\al=0.5$ & $L^2$-norm  & 7.12e-5 &1.76e-5  &4.37e-6  &1.10e-6  &2.76e-7 & $O(h^{2.00})$  \\
     \cline{2-8}
     & $H^1$-norm          & 5.53e-2 & 3.93e-2 & 2.79e-2 & 1.97e-3 & 1.40e-3  & $O(h^{0.50})$   \\
     \hline
     $\al=0.95$ & $L^2$-norm   & 4.91e-5 &1.21e-5  &3.03e-6  &7.66e-7  &1.96e-7 & $O(h^{2.02})$\\
     \cline{2-8}
     & $H^1$-norm           & 5.60e-2 &3.96e-2 & 2.80e-2 & 1.98e-3 & 1.40e-3& $O(h^{0.50})$\\
     \hline
     \end{tabular}
\end{center}
\end{table}

\begin{figure}[h!]
\center
  \includegraphics[trim = 1cm .1cm 2cm 0.5cm, clip=true,width=10cm]{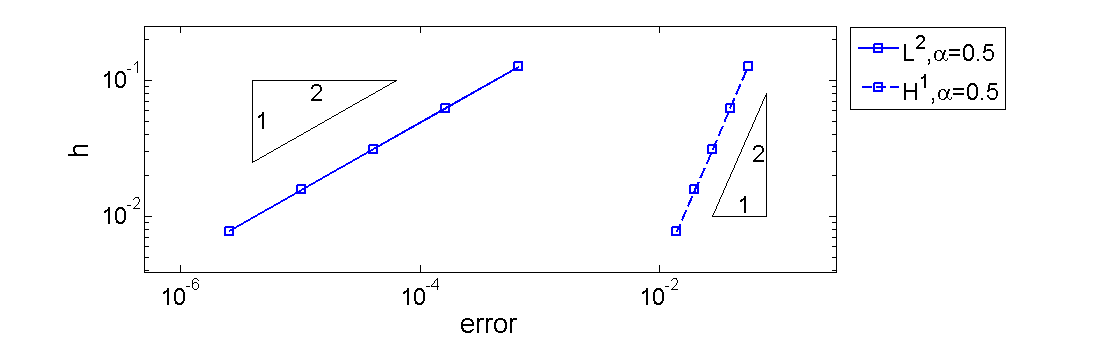}
  \caption{Numerical results, i.e., errors $\| u(t)-\luh(t)\|_{\dH p}$, $p=0,1$, $t=1$,
  for example (1b) (Dirac $\delta$-function) with $\alpha=0.5$. The Dirac $\delta$-function
  is aligned with grid points.}\label{fig:weak1Dsuper}
\end{figure}

\subsubsection{Numerical results for example (1c)}
Here the coefficient is variable, and the discrete solution $\luh$ does not
have a convenient explicit representation. Hence, we apply the fully discrete scheme with
finite difference discretization on time by \eqref{eqn:timedisc}.
The benchmark solution $u$ is computed on a very fine mesh, i.e., with a spacial mesh size
$h=1/512$ and a time step size $\tau=1.0 \times 10^{-5}$. The
$L^2(\Omega)$- and $H^1(\Omega)$-norms of the error are reported in Table \ref{tab:variable} at $t=1$ for $\al=0.5$.
The results confirm the discussions in Section \ref{general_problems}.

\begin{table}[h!]
\caption{Numerical results, i.e., errors $\|u(t)-\luh(t)\|_{\dH p}$, $p=0,1$, at $t=1$, for
example (1c) (variable coefficients) with $\al=0.5$ and
mesh sizes $1/2^k$.}\label{tab:variable}
\begin{center}
     \begin{tabular}{|c|ccccc|c|}
     \hline
     $h$ & $1/8$ & $1/16$ & $1/32$& $1/64$& $1/128$& rate \\
     \hline
     $L^2$-error & 8.42e-4 & 2.14e-4 & 5.38e-5 & 1.34e-5 & 3.30e-6 & $O(h^{2.00})$ \\
     $H^1$-error & 1.76e-2 & 8.89e-3 & 4.45e-3 & 2.21e-3 & 1.08e-3 & $O(h^{1.01})$ \\
     \hline
     \end{tabular}
\end{center}
\end{table}

\subsection{2-d Examples}\label{ssect:2D}
Now we consider \eqref{eq1} for $d=2$ on the unit square $\Om=[0,1]^2$
for the following data:
\begin{enumerate}
\item[(2a)]  Nonsmooth data: $ f(x,t) = (\chi_{[1/2,1]}(t)+1)\chi_{[1/4,3/4]\times[1/4,3/4]}$.

\item[(2b)] Very weak data: $f(x,t)=(\chi_{[1/2,1]}(t)+1)\delta_\Gamma $ with $\Gamma$ being
the boundary of the square $[1/4,3/4]\times[1/4,3/4]$ with $\langle \delta_\Gamma,\phi\rangle
= \int_\Gamma \phi(s) ds$. One may view $(v,\chi)$ for $\chi \in X_h \subset\dot H^{1/2+\epsilon}(\Om)$
as duality pairing between the spaces $H^{-1/2-\epsilon}(\Om)$ and $\dot H^{1/2+\epsilon}(\Om)$
for any $\epsilon >0$ so that $\delta_\Gamma \in H^{-1/2-\epsilon}(\Omega)$. Indeed, it follows
from H\"{o}lder's inequality and the continuity of the trace operator
from $\dot H^{{1/2}+\epsilon}(\Omega)$ to $L^2(\Gamma)$ \cite{AdamsFournier:2003}
 that
\begin{equation*}
  \begin{aligned}
   \|\delta_\Gamma\|_{H^{-1/2-\epsilon}(\Omega)}&= \sup_{\phi \in\dot H^{1/2+\epsilon}(\Omega)}
   \frac{|\int_\Gamma \phi(s)ds|}{\|\phi\|_{\dot H^{1/2+\epsilon}(\Omega)}}
   \le  |\Gamma|^{1/2} \sup_{\phi \in\dot H^{1/2+\epsilon} (\Omega)}
      \frac{\|\phi\|_{L^2(\Gamma)}}{\|\phi\|_{\dot H^{1/2+\epsilon}(\Omega)}} \le C.
 \end{aligned}
 \end{equation*}
\end{enumerate}
To discretize the problem, we divide $(0,1)$ into
$N=2^k$ equally spaced subintervals with a mesh size $h=1/N$ so that unit square $(0,1)^2$ is divided into $N^2$
small squares. We get a symmetric triangulation of the domain $(0,1)^2$ by connecting the
diagonal of each small square.
Therefore, the lumped mass method and standard Galerkin method have the same
convergence rates. On these meshes, $\bar \la^h_{n,m}$ and
$\bar \fy^h_{n,m}$, $1 \le n,m \le N-1$,
i.e., eigenvalues and eigenvectors
of the discrete Laplacian $\bar \Delta_h$, are explicitly given by:
\begin{equation*}
    \bar\la^h_{n,m}= \frac{4}{h^2} \left(\sin^2\frac{n\pi h}{2} + \sin^2\frac{m\pi h}{2} \right), \quad
    \bar\fy^h_{n,m}(x_i,y_j)= 2\sin(n\pi x_i)\sin(m\pi y_j).
\end{equation*}
respectively, where $(x_i,y_j)$, $i,j=1,\ldots,N-1$, is a mesh point. Then the semidiscrete approximation can
be computed via the explicit representation \eqref{eqn:Duhamel_oo}.

\subsubsection{Numerical results for example (2a)}
In this example the right hand side data $f(x,t)$ is in the space
$L^{\infty}(0,1;\dot H^{1/2-\ep}(\Om))$ with $\ep > 0$, and
the numerical results were computed at $t=1$ for $\al=0.1,0.5$ and $0.95$;
see Table \ref{tab:nonsmooth2D} and Figure \ref{fig:nonsmooth2D}.
The slopes of the error curves in a log-log
scale are $2$ and $1$, respectively, for $L^2(\Omega)$- and $H^1(\Omega)$-norm of the errors,
which agrees well with the theoretical results for the nonsmooth case.
\begin{table}[!ht]
\caption{Numerical results, i.e., errors $\|u(t)-\luh(t) \|_{\dH p}$, $p=0,1$ at $t=1$,
for example (2a) (nonsmooth data) with  mesh sizes $h=1/2^k$ in either direction.}
\label{tab:nonsmooth2D}
\begin{center}
     \begin{tabular}{|c|c|c|c|c|c|c|c|}
     \hline
      $\al$ & $k$ & $3$ & $4$ & $5$ &$6$ & $7$ & rate\\
     \hline
     $\al=0.1$ & $L^2$-norm & 9.66e-4 & 2.48-4  & 6.26e-5 &  1.57e-5 &  3.93e-6  & $O(h^{1.99})$ \\
     \cline{2-8}
     & $H^1$-norm           & 2.06e-2   &1.04e-2   &5.24e-3  &2.63e-3   &1.31e-3 & $O(h^{0.99})$  \\
     \hline
     $\al=0.5$ & $L^2$-norm  & 9.82e-4 & 2.52-4  & 6.36e-5 &  1.59e-5 &  3.99e-6 & $O(h^{1.99})$  \\
     \cline{2-8}
     & $H^1$-norm           & 2.10e-2   &1.07e-2   &5.35e-3  &2.68e-3   &1.34e-3 & $O(h^{0.99})$   \\
     \hline
     $\al=0.95$ & $L^2$-norm   & 9.82e-4 & 2.52-4  & 6.36e-5 &  1.61e-5 &  4.02e-6  & $O(h^{1.99})$\\
     \cline{2-8}
     & $H^1$-norm           &2.13e-2   &1.08e-2   &5.42e-3  &2.71e-3   &1.36e-3& $O(h^{0.99})$\\
     \hline
     \end{tabular}
\end{center}
\end{table}

\begin{figure}[h!]
\center
  \includegraphics[trim = 1cm .1cm 2cm 0.5cm, clip=true,width=10cm]{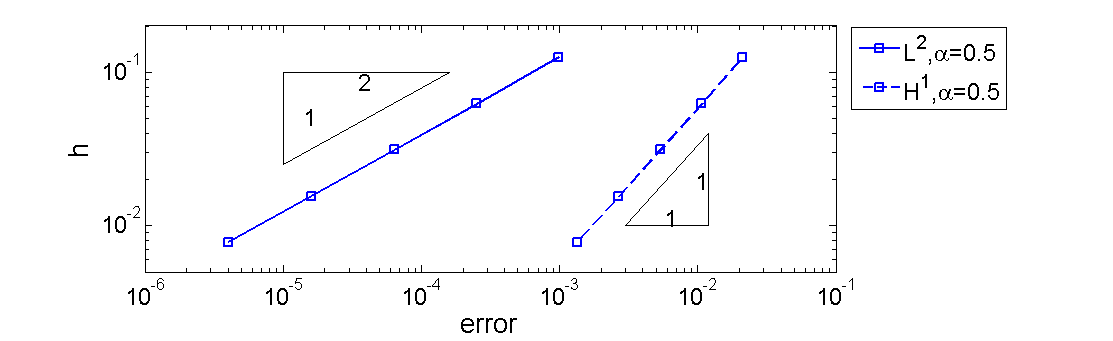}
  \caption{Numerical results, i.e., errors $\|u(t)-\luh(t)\|_{\dH p}$, $p=0,1$, at $t=1$,
   for example (2a) (nonsmooth data) with $\alpha=0.5$.}\label{fig:nonsmooth2D}
\end{figure}

\subsubsection{Numerical results for example (2b)}
In Table \ref{tab:weak2D}, we present the $L^2(\Omega)$- and $H^1(\Omega)$-norms of the error for this example.
The $H^1(\Omega)$-norm of the error decays at the theoretical rate, however the $L^2(\Omega)$-norm of
the error exhibits superconvergence. This is attributed to the fact that the boundary $\Gamma$
is fully aligned with element edges. In contrast, if we choose $P_h f$ as the discrete right
hand side for the lumped mass semidiscrete problem instead of $\bar P_h f$, then
the $L^2(\Omega)$-norm of the error converges only at the standard order; see Table \ref{tab:weak2D2}.

\begin{table}[!ht]
\caption{Numerical results, i.e., errors $\|u(t)-\luh(t)\|_{\dH p}$, $p=0,1$ at $t=1$, for
example (2b) (Dirac $\delta$-function) with mesh sizes $h=1/2^k$ in either direction.}
\label{tab:weak2D}
\begin{center}
     \begin{tabular}{|c|c|c|c|c|c|c|c|}
     \hline
      $\al$ & $k$ & $3$ & $4$ & $5$ &$6$ & $7$ & rate\\
     \hline
     $\al=0.1$ & $L^2$-norm & 4.61e-3 & 1.25e-3  & 3.31e-4 &  8.68e-5 &  2.39e-5  & $O(h^{1.90})$ \\
     \cline{2-8}
     & $H^1$-norm           & 1.60e-1   &9.92e-2   &6.43e-2  &4.43e-2   &3.16e-2 & $O(h^{0.58})$  \\
     \hline
     $\al=0.5$ & $L^2$-norm  & 4.67e-3 & 1.26e-3  & 3.34e-4 &  8.76e-5 &  2.40e-5   & $O(h^{1.91})$  \\
     \cline{2-8}
     & $H^1$-norm          & 1.60e-1   &9.92e-2   &6.44e-2  &4.50e-2   &3.17e-2 & $O(h^{0.58})$   \\
     \hline
     $\al=0.95$ & $L^2$-norm  & 4.70e-3 & 1.27e-3  & 3.36e-4 &  8.81e-5 &  2.42e-5    & $O(h^{1.91})$\\
     \cline{2-8}
     & $H^1$-norm          & 1.61e-1   &9.98e-2   &6.46e-2  &4.50e-2   &3.17e-2& $O(h^{0.58})$\\
     \hline
     \end{tabular}
\end{center}
\end{table}

\begin{table}[!ht]
\caption{Numerical results, i.e., errors $\|u(t)-\luh(t)\|_{\dH p}$, $p=0,1$, at $t=1$, for example (2b)
(Dirac $\delta$-function) with $f_h=P_h f$ and mesh sizes $h=1/2^k$ in either direction.}
\label{tab:weak2D2}
\begin{center}
     \begin{tabular}{|c|c|c|c|c|c|c|c|}
     \hline
      $\al$ & $k$ & $3$ & $4$ & $5$ &$6$ & $7$ & rate\\
     \hline
     $\al=0.5$ & $L^2$-norm  & 1.19e-2 & 4.55e-3  & 1.69e-3 &  6.15e-4 &  2.22e-4   & $O(h^{1.44})$  \\
     \cline{2-8}
     & $H^1$-norm          & 3.28e-1   &2.32e-1   &1.67e-1  &1.13e-1   &8.21e-2 & $O(h^{0.50})$   \\
     \hline
     \end{tabular}
\end{center}
\end{table}

\section{conclusion}

In this paper, we have studied two semidiscrete finite element schemes, i.e.,
the semidiscrete Galerkin method and the lumped mass method, for
the time-fractional diffusion problem with a nonsmooth right hand side $f\in L^\infty(0,T;\dH q)$, with $-1<q\leq 1$. We
derived almost optimal estimates of the error and its gradient for the Galerkin method, and also almost optimal
estimates of the gradient of the error for the lumped mass method. Optimal estimates  for the $L^2(\Omega)$-norm of the error
can only be shown under certain restrictions on the mesh such that condition \eqref{eqn:condQ} is fulfilled.
There are several possible extensions of the work. First, the error estimates are expected to be useful for
analyzing relevant inverse problems, which will be studied in a future work. Second, it is natural to study the fully
discrete scheme, e.g., with finite difference or discontinuous Galerkin method in time.

\section*{Acknowledgments}
The research of B. Jin has been support by NSF Grant DMS-1319052, R. Lazarov and Z. Zhou was supported in
parts by US NSF Grant DMS-1016525 and J. Pasciak has been supported by NSF Grant
DMS-1216551. The work of all authors has been  supported in parts also  by Award No.
KUS-C1-016-04, made by King Abdullah University of Science and Technology (KAUST).

\bibliographystyle{abbrv}
\bibliography{frac}

\begin{thebibliography}{10}

\bibitem{AdamsGelhar:1992}
E.~E. Adams and L.~W. Gelhar.
\newblock Field study of dispersion in a heterogeneous aquifer: 2. spatial
  moments analysis.
\newblock {\em Water Res. Research}, 28(12):3293–--3307, 1992.

\bibitem{AdamsFournier:2003}
R.~A. Adams and J.~J.~F. Fournier.
\newblock {\em Sobolev {S}paces}.
\newblock Elsevier/Academic Press, Amsterdam, second edition, 2003.

\bibitem{Berkowitz:2002}
B.~Berkowitz, J.~Klafter, R.~Metzler, and H.~Scher.
\newblock Physical pictures of transport in heterogeneous media:
  Advection-dispersion, random-walk, and fractional derivative formulations.
\newblock {\em Water Res. Research}, 38(10):9--1--9--12, 2002.

\bibitem{chatzipa-l-thomee12}
P.~Chatzipantelidis, R.~Lazarov, and V.~Thom{\'e}e.
\newblock Some error estimates for the lumped mass finite element method for a
  parabolic problem.
\newblock {\em Math. Comp.}, 81(277):1--20, 2012.

\bibitem{ChengNakagawaYamamoto:2009}
J.~Cheng, J.~Nakagawa, M.~Yamamoto, and T.~Yamazaki.
\newblock Uniqueness in an inverse problem for a one-dimensional fractional
  diffusion equation.
\newblock {\em Inverse Problems}, 25(11):115002, 1--16, 2009.

\bibitem{GionaCerbeliRoman:1992}
M.~Giona, S.~Cerbelli, and H.~E. Roman.
\newblock Fractional diffusion equation and relaxation in complex viscoelastic
  materials.
\newblock {\em Phys. A}, 191(1--4):449--453, 1992.

\bibitem{HatanoHatano:1998}
Y.~Hatano and N.~Hatano.
\newblock Dispersive transport of ions in column experiments: An explanation of
  long-tailed profiles.
\newblock {\em Water Res. Research}, 34(5):1027--1033, 1998.

\bibitem{JinLazarovPasciakZhou:2013}
B.~Jin, R.~Lazarov, J.~Pasciak, and Z.~Zhou.
\newblock Galerkin fem for fractional order parabolic equations with initial
  data in {$H^{-s},~0\le s \le 1$}.
\newblock Proc. 5th Conf. Numer. Anal. Appl. (June 15-20, 2012), Springer, in
  press, 2013.

\bibitem{Bangti_LZ_2013}
B.~Jin, R.~Lazarov, and Z.~Zhou.
\newblock Error estimates for a semidiscrete finite element method for
  fractional order parabolic equations.
\newblock {\em SIAM J. Numer. Anal.}, 51(1):445--466, 2013.

\bibitem{JinRundell:2012b}
B.~Jin and W.~Rundell.
\newblock An inverse problem for a one-dimensional time-fractional diffusion
  equation.
\newblock {\em Inverse Problems}, 28(7):075010, 19 pp., 2012.

\bibitem{KeungZou:1998}
Y.~L. Keung and J.~Zou.
\newblock Numerical identifications of parameters in parabolic systems.
\newblock {\em Inverse Problems}, 14(1):83--100, 1998.

\bibitem{KilbasSrivastavaTrujillo:2006}
A.~Kilbas, H.~Srivastava, and J.~Trujillo.
\newblock {\em Theory and {A}pplications of {F}ractional {D}ifferential
  {E}quations}.
\newblock Elsevier, Amsterdam, 2006.

\bibitem{LiXu:2009}
X.~Li and C.~Xu.
\newblock A space-time spectral method for the time fractional diffusion
  equation.
\newblock {\em SIAM J. Numer. Anal.}, 47(3):2108--2131, 2009.

\bibitem{LinXu:2007}
Y.~Lin and C.~Xu.
\newblock Finite difference/spectral approximations for the time-fractional
  diffusion equation.
\newblock {\em J. Comput. Phys.}, 225(2):1533--1552, 2007.

\bibitem{McLean-Thomee-IMA}
W.~{McLean} and V.~Thom{\'e}e.
\newblock Maximum-norm error analysis of a numerical solution via {L}aplace
  transformation and quadrature of a fractional-order evolution equation.
\newblock {\em IMA J. Numer. Anal.}, 30(1):208--230, 2010.

\bibitem{McLeanThomee:2010}
W.~{McLean} and V.~Thom{\'e}e.
\newblock Numerical solution via {L}aplace transforms of a fractional order
  evolution equation.
\newblock {\em J. Int. Eq. Appl.}, 22(1):57--94, 2010.

\bibitem{MontrollWeiss:1965}
E.~W. Montroll and G.~H. Weiss.
\newblock Random walks on lattices. {II}.
\newblock {\em J. Math. Phys.}, 6(2):167--181, 1965.

\bibitem{Mustapha:2011}
K.~Mustapha.
\newblock An implicit finite-difference time-stepping method for a
  sub-diffusion equation, with spatial discretization by finite elements.
\newblock {\em IMA J. Numer. Anal.}, 31(2):719--739, 2011.

\bibitem{Nigmatulin}
R.~Nigmatulin.
\newblock The realization of the generalized transfer equation in a medium with
  fractal geometry.
\newblock {\em Phys. Stat. Sol. B}, 133:425--430, 1986.

\bibitem{Podlubny_book}
I.~Podlubny.
\newblock {\em Fractional {D}ifferential {E}quations}.
\newblock Academic Press, San Diego, CA, 1999.

\bibitem{Sakamoto_2011}
K.~Sakamoto and M.~Yamamoto.
\newblock Initial value/boundary value problems for fractional diffusion-wave
  equations and applications to some inverse problems.
\newblock {\em J. Math. Anal. Appl.}, 382(1):426--447, 2011.

\bibitem{Seybold:2008}
H.~Seybold and R.~Hilfer.
\newblock Numerical algorithm for calculating the generalized
  {M}ittag-{L}effler function.
\newblock {\em SIAM J. Numer. Anal.}, 47(1):69--88, 2008/09.

\bibitem{Thomee97}
V.~Thom{\'e}e.
\newblock {\em Galerkin {F}inite {E}lement {M}ethods for {P}arabolic
  {P}roblems}, volume~25 of {\em Springer Series in Computational Mathematics}.
\newblock Springer-Verlag, Berlin, 1997.

\bibitem{XieZou:2006}
J.~Xie and J.~Zou.
\newblock Numerical reconstruction of heat fluxes.
\newblock {\em SIAM J. Numer. Anal.}, 43(4):1504--1535, 2006.

\bibitem{YusteAcedo:2005}
S.~Yuste and L.~Acedo.
\newblock An explicit finite difference method and a new von {N}eumann-type
  stability analysis for fractional diffusion equations.
\newblock {\em SIAM J. Numer. Anal.}, 42(5):1862--1874, 2005.

\end{thebibliography}
\end{document}